\documentclass[a4paper,10pt]{article}
\usepackage{cmap}
\usepackage[utf8]{inputenc}
\usepackage[english]{babel}
\usepackage{amsfonts}
\usepackage{amssymb, amsthm}
\usepackage{xcolor}
\usepackage{marginnote}
\usepackage{amsmath}
\usepackage{a4wide}
\usepackage[affil-it]{authblk}

\numberwithin{equation}{section}

\usepackage[colorlinks=true,linktocpage,pdfpagelabels,
bookmarksnumbered,bookmarksopen]{hyperref}
\definecolor{ForestGreen}{rgb}{0.1,0.6,0.05}
\definecolor{EgyptBlue}{rgb}{0.063,0.1,0.6}
\hypersetup{
	colorlinks=true,
	linkcolor=EgyptBlue,         
	citecolor=ForestGreen,
	urlcolor=olive
}

\usepackage[hyperpageref]{backref}

\usepackage{accents}

\def\wolp{\accentset{\circ}{W}_p^{1}}

\newtheorem{thm}{Theorem}[section]
\newtheorem{lemma}[thm]{Lemma}
\newtheorem{proposition}[thm]{Proposition}
\newtheorem{cor}[thm]{Corollary}
\theoremstyle{definition}
\newtheorem{remark}[thm]{Remark}

\title{On qualitative properties of solutions for elliptic problems with the $p$-Laplacian through domain perturbations
	\\ \medskip}

\author{Vladimir Bobkov\thanks{E-mail: \texttt{bobkov@kma.zcu.cz}}~}
\affil{{\small Department of Mathematics and NTIS, Faculty of Applied Sciences, University of West Bohemia, Univerzitn\'i 8, 301 00 Plze\v{n}, Czech Republic. Institute of Mathematics, Ufa Federal Research Centre, RAS, Chernyshevsky str.~112, 450008 Ufa, Russia}}
\author{Sergey Kolonitskii\thanks{E-mail: \texttt{sergey.kolonitskii@gmail.com}}}
\affil{{\small Saint Petersburg Electrotechnical University "LETI", 5 Professora Popova st., St.~Petersburg, 197376 Russia}}
\date{}

\begin{document}
\maketitle

\begin{abstract}
	We study the dependence of least nontrivial critical levels of the energy functional corresponding to the zero Dirichlet problem $-\Delta_p u = f(u)$ in a bounded domain $\Omega \subset \mathbb{R}^N$ upon domain perturbations. 
	Assuming that the nonlinearity $f$ is superlinear and subcritical, we establish Hadamard-type formulas for such critical levels.
	As an application, we show that among all (generally eccentric) spherical annuli $\Omega$ least nontrivial critical levels attain maximum if and only if $\Omega$ is concentric. 
	As a consequence of this fact, we prove the nonradiality of least energy nodal solutions whenever $\Omega$ is a ball or concentric annulus. 

	
\end{abstract}

\section{Introduction and main results}\label{sec:intro}

Let $\Omega$ be a bounded domain in $\mathbb{R}^N$, $N \geqslant 2$, with the boundary $\partial \Omega$ of class $C^{2,\varsigma}$, $\varsigma \in (0,1)$.
Consider the boundary value problem
\begin{equation}\label{D}
	\tag{$\mathcal{D}$}
	\left\{
	\begin{aligned}
		-\Delta_p u &= f(u) 
		&&\text{in } \Omega, \\
		u &= 0  &&\text{on } \partial \Omega,
	\end{aligned}
	\right.
\end{equation}	
where $\Delta_p u := \text{div}(|\nabla u|^{p-2} \nabla u)$ is the $p$-Laplacian, $p>1$. 
Denote $p^* = \frac{Np}{N-p}$ if $p < N$ and $p^* = +\infty$ if $p \geqslant N$.
We will always impose the following assumptions on the nonlinearity $f: \mathbb{R} \to \mathbb{R}$:

\begin{itemize}\addtolength{\itemindent}{1em}
	\item[$(A_1)$] $f \in C^1(\mathbb{R} \setminus \{0\}) \cap C^{0,\gamma}_{\text{loc}}(\mathbb{R})$ for some $\gamma \in (0,1)$.
	\item[$(A_2)$] There exist $q \in (p, p^*)$ and $C>0$ such that $|s f'(s)|, |f(s)| \leqslant C (|s|^{q-1}+1)$ for all $s \in \mathbb{R}\setminus\{0\}$.\footnote{If $p \geqslant N$, this assumption can be relaxed, see \cite[condition (\textbf{F4}) and Lemma~5.6]{Nazar}.}
	\item[$(A_3)$] $f'(s) > (p-1)\dfrac{f(s)}{s} > 0$ for all $s \in \mathbb{R}\setminus \{0\}$, and $\limsup\limits_{s \to 0} \dfrac{f(s)}{|s|^{p-2}s} < \lambda_p(\Omega)$, where
	\begin{equation}\label{mu:lambda}
		\lambda_p(\Omega) := \min_{u \in \wolp(\Omega) \setminus \{0\}} \frac{\int_\Omega |\nabla u|^p \, dx}{\int_\Omega |u|^p \, dx}
	\end{equation}
	is the first eigenvalue of the zero Dirichlet $p$-Laplacian in $\Omega$.
	\item[$(A_4)$] There exist $s_0 > 0$ and $\theta > p$ such that $0 < \theta F(s) \leqslant s f(s)$ for all $|s| > s_0$, where 
	$$
	F(s) := \int_0^s f(t) \, dt.
	$$
\end{itemize}

Problem \eqref{D} corresponds to the energy functional $E: \wolp(\Omega) \to \mathbb{R}$ defined as
$$
E[u] = \frac{1}{p} \int_\Omega |\nabla u|^p \, dx - \int_\Omega F(u) \, dx.
$$
The functional $E$ is weakly lower semicontinuous and belongs to $C^1(\wolp(\Omega))$. 
By definition, a weak solution of \eqref{D} is a critical point of $E$. 
Moreover, any weak solution of \eqref{D} is $C^{1,\beta}(\overline{\Omega})$-smooth, $\beta \in (0,1)$, and any constant-sign weak solution satisfies the Hopf maximum principle\footnote{see Remarks~\ref{rem:regularity} and \ref{rem:maximum} in Appendix~\ref{section:appendix}.}. 

If for some $c \in \mathbb{R}$ there exists a nontrivial critical point $u$ of $E$ such that $E[u] = c$, then $c$ is called a nontrivial critical level of~$E$. 
We are interested in \textit{least} nontrivial critical levels $\mu_+(\Omega)$ and $\mu_-(\Omega)$ among positive and negative solutions of \eqref{D}, respectively. 
In Appendix~\ref{section:appendix} below we discuss that under assumptions $(A_1)-(A_4)$, $\mu_+(\Omega)$ and $\mu_-(\Omega)$ can be defined as
\begin{equation}\label{mu:Nehari}
\mu_+(\Omega) = \min_{v \in \mathcal{N}(\Omega), \, v \geqslant 0} E[v] 
\quad \text{and} \quad 
\mu_-(\Omega) = \min_{v \in \mathcal{N}(\Omega), \, v \leqslant 0} E[v],
\end{equation}
where
$$
\mathcal{N}(\Omega) := \{u \in \wolp(\Omega) \setminus \{0\}:\, E'[u]u \equiv \int_\Omega |\nabla u|^p \, dx - \int_\Omega u \, f(u) \, dx = 0\}
$$
is the Nehari manifold. 
Minimizers of \eqref{mu:Nehari} exist and they are least energy constant-sign solutions of \eqref{D}. Moreover, $\mu_\pm(\Omega) > 0$.

The first goal of the present article is to study the behavior of $\mu_\pm(\Omega)$ under smooth domain perturbations $\Omega_t := \Phi_t(\Omega)$ driven by a family of diffeomorphisms
\begin{equation}\label{Phi}
\Phi_t(x) = x + t R(x), 
\quad 
R \in C^1(\mathbb{R}^N, \mathbb{R}^N), 
\quad 
|t|<\delta,
\end{equation}
where $\delta > 0$ is small enough.
Let us take an arbitrary minimizer $v_0$ of $\mu_+(\Omega)$ and consider a function $v_t(y) := v_0(\Phi_t^{-1}(y))$, $y \in \Omega_t$. It is not hard to see that $v_t \in \wolp(\Omega_t)$ and $v_t \geqslant 0$ on $\Omega_t$.
By Lemma~\ref{lem:superlin} and Remark~\ref{rem:Omega_t} from Appendix~\ref{section:appendix}, for each $|t| < \delta$ we can find a unique constant $\alpha(v_t) > 0$ such that $\alpha(v_t) v_t \in \mathcal{N}(\Omega_t)$. 
Consequently, $\mu_+(\Omega_t) \leqslant E[\alpha(v_t) v_t]$. (We always assume by default that domains of integration in $E[\alpha(v_t) v_t]$ are $\Omega_t$.) Analogous facts remain valid if we take any minimizer $w_0$ of $\mu_-(\Omega)$ and consider $w_t(y) := w_0(\Phi_t^{-1}(y))$, $y \in \Omega_t$.

We prove the following result.
\begin{thm}\label{thm:1}
	Assume that $(A_1)-(A_4)$ are satisfied. 
	Then $\mu_+(\Omega_t)$ and $\mu_-(\Omega_t)$ are continuous at $t=0$. Moreover, 
	$E[\alpha(v_t) v_t]$ and $E[\alpha(w_t) w_t]$ are differentiable with respect to $t \in (-\delta, \delta)$ and the following Hadamard-type formulas hold:
	\begin{align}
	\notag
	\left.\frac{\partial E[\alpha(v_t) v_t]}{\partial t}\right|_{t=0} &= - \frac {p-1} p \int_{\partial \Omega} \left| \frac{\partial v_0}{\partial n} \right|^p \left<R, n\right> \,d\sigma, \\
	\label{mu:Neh2}
	\left.\frac{\partial E[\alpha(w_t) w_t]}{\partial t}\right|_{t=0} &= - \frac {p-1} p \int_{\partial \Omega} \left| \frac{\partial w_0}{\partial n} \right|^p \left<R, n\right> \,d\sigma,
	\end{align}
	where $n$ is the outward unit normal vector to $\partial \Omega$ and $\left<\cdot,\cdot\right>$ stands for the scalar product in $\mathbb{R}^N$.
\end{thm}

\begin{remark}
	It is rather counterintuitive that the domain derivative does not explicitly depend on the weak term $f$. 
\end{remark}

Origins of this problematic go back to the work of Hadamard~\cite{hadamard}, where he proved that the first eigenvalue $\lambda_2(\Omega_t)$ of the zero Dirichlet Laplace operator in $\Omega_t$ is differentiable at $t=0$ and deduced its expression (see \eqref{hadamar_eigen} below with $p=2$) which nowadays is known as the \textit{Hadamard formula}. We refer the reader, for instance, to \cite{sokolowski,henrotpiere,delfour} for the general theory of the shape optimization and related historical remarks. The first eigenvalue \eqref{mu:lambda} in the general case $p>1$ was treated in \cite{garcia} (see also~\cite{lamberti}), and it was proved that 
\begin{equation}\label{hadamar_eigen}
\left.\frac{\partial \lambda_p(\Omega_t)}{\partial t}\right|_{t=0}  = -(p-1) \int_{\partial \Omega} \left|\frac{\partial \varphi_p}{\partial n}\right|^p \, \left<R, n\right> \, d\sigma,
\end{equation}
where $\varphi_p$ is the eigenfunction associated with $\lambda_p(\Omega)$ and normalized such that $\|\varphi_p\|_{L^p(\Omega)} = 1$. 
At the same time, in Remark~\ref{rem:nondiff} below we discuss that $\mu_+(\Omega_t)$ and $\mu_-(\Omega_t)$ are not differentiable at $t=0$, in general.

Note that the main prototypical nonlinearity for \eqref{D} is given by $f(u) = |u|^{q-2} u$, where $q \in (p, p^*)$. It can easily be checked that assumptions $(A_1) - (A_4)$ are satisfied. Due to the homogeneity and oddness of $f$, the problem of finding the least critical levels $\mu_\pm(\Omega)$ can be rewritten in the form of the nonlinear Rayleigh quotient 
\begin{equation}\label{mu:Hom}
\mu_q(\Omega) = 
\min_{u \in \wolp(\Omega) \setminus \{0\}} J(u) 
:=
\min_{u \in \wolp(\Omega) \setminus \{0\}}
\frac{\int_\Omega |\nabla u|^p \, dx}{\left(\int_\Omega |u|^q \, dx\right)^\frac{p}{q}}.
\end{equation}
The minimum is achieved, and, after an appropriate normalization, corresponding minimizers satisfy \eqref{D}. These facts remain valid for all $q \in [1, p^*)$.

As a corollary of the proof of Theorem~\ref{thm:1} we obtain the following fact.
\begin{thm}\label{thm:1.5}
	Let $q \in [1,p^*)$. Then $\mu_q(\Omega_t)$ is continuous at $t=0$. Moreover, if $u_0$ is a minimizer of $\mu_q(\Omega)$ normalized such that $\|u_0\|_{L^q(\Omega)} = 1$, and $u_t(y) := u_0(\Phi_t^{-1}(y))$, $y \in \Omega_t$, then $J(u_t)$ is differentiable with respect to $t \in (-\delta, \delta)$ and 
	\begin{equation*}
	\left.\frac{\partial J(u_t)}{\partial t} \right|_{t=0}= - (p-1) \int_{\partial \Omega} \left| \frac{\partial u_0}{\partial n} \right|^p \left<R, n\right> \,d\sigma.
	\end{equation*}
\end{thm}
 
\medskip
The second aim of our work is to use Theorem~\ref{thm:1} for studying a shape optimization problem for $\mu_\pm(\Omega)$ over a special class of domains. Namely, let $\Omega$ be an open spherical annulus $B_{R_1}(x) \setminus \overline{B_{R_0}(y)}$, where $|x-y| < R_1 - R_0$. Due to the invariance of \eqref{D} upon orthogonal transformations of coordinates, we can take $x=0$ and $y = se_1$, where $s \in [0, R_1-R_0)$ and $e_1$ is the first coordinate vector. 
For simplicity and to avoid ambiguity, we denote 
\begin{align*}
\tilde{\mu}_\pm(s) &:= \mu_\pm(B_{R_1}(0) \setminus \overline{B_{R_0}(se_1)}).
\end{align*}
In order to guarantee the existence of minimizers of $\tilde{\mu}_\pm(s)$ for all $s \in [0, R_1-R_0)$ (see Appendix~\ref{section:appendix}) the second part of assumption $(A_3)$ must be satisfied for any annulus $B_{R_1}(0) \setminus \overline{B_{R_0}(se_1)}$, $s \in [0, R_1-R_0)$. 
To this end we impose the following additional assumption (see a discussion below):
\begin{itemize}\addtolength{\itemindent}{1em}
	\item[$(A_3^*)$] $\limsup\limits_{s \to 0} \dfrac{f(s)}{|s|^{p-2}s} < \lambda_p(B_{R_1}(0) \setminus \overline{B_{R_0}((R_1-R_0)e_1)})$.
\end{itemize}

We consider the following question: 
\begin{center}
\textit{Which value of the displacement $s \in [0, R_1-R_0)$ maximizes/minimizes $\tilde{\mu}_\pm(s)$?}
\end{center}

In the case of the first eigenvalue \eqref{mu:lambda} this question was addressed in several articles, see \cite{hersh,ramm,harrel,kesavan} for the linear case $p=2$, and \cite{chorwad,anoopbobsasi} for general $p>1$. 
It was proved in \cite[Theorem~1.1]{anoopbobsasi} that $\lambda_p(s) := \lambda_p(B_{R_1}(0) \setminus \overline{B_{R_0}(se_1)})$ is continuous and strictly decreasing on $[0, R_1-R_0)$, which implies that $\lambda_p(s)$ attains its maximum if and only if $s=0$ and attains its minimum if and only if $s = R_1-R_0$. These facts justify the choice $s=R_1-R_0$ in assumption~$(A_3^*)$. 

The common approach to prove the monotonicity of $\lambda_p(s)$ is to consider a perturbation $\Phi_t$ which shifts the inner boundary $\partial B_{R_0}(se_1)$ along the direction $e_1$ while the outer boundary $\partial B_{R_1}(0)$ remains fixed. Then the Hadamard formula \eqref{hadamar_eigen} allows to find a derivative of $\lambda_p(s)$ with respect to the displacement $s$ in terms of an integral over the inner boundary. 
Hence, to show that $\lambda_p'(s) < 0$, one can try to compare values of the normal derivatives of the  eigenfunction of $\lambda_p(s)$ on hemispheres $\{x \in \partial B_{R_0}(se_1):\, x_1 < s\}$ and $\{x \in \partial B_{R_0}(se_1):\, x_1 > s\}$. In the linear case $p=2$, reflection arguments together with the strong comparison principle can be applied to show that such values are strictly ordered, which leads to $\lambda_p'(s) < 0$ for any $s \in (0, R_1-R_0)$. (Note that $\lambda_p'(0) = 0$ due to symmetry reasons.)
At the same time, the lack of strong comparison principles for the general nonlinear case $p>1$ entails the use of additional arguments. In \cite{chorwad}, applying an appropriate version of the weak comparison principle, it was shown that $\lambda_p'(s) \leqslant 0$ for all $s \in (0, R_1-R_0)$. The strict negativity of $\lambda_p'(s)$ was obtained recently in \cite{anoopbobsasi} bypassing the usage of (global) strong comparison results. 

Considering the least nontrivial critical levels $\tilde{\mu}_\pm(s)$, we follow the strategy described above. 
To this end, we employ two symmetrization methods: \textit{polarization} \cite{brocksol,BartschWethWillem} and \textit{spherical symmetrization} (i.e., \textit{foliated Schwarz symmetrization}) \cite{kawohl,BartschWethWillem}. 
With the help of these methods, we use the ideas from \cite{anoopbobsasi}, to derive the following result.
\begin{thm}\label{thm:2}
	Assume that $(A_1)-(A_4)$ and $(A_3^*)$ are satisfied. 
	Then $\tilde{\mu}_+(s)$ and $\tilde{\mu}_-(s)$ are continuous and strictly decreasing on $[0, R_1-R_0)$.
\end{thm}

As a corollary of the proof of Theorem~\ref{thm:2} we have the following fact which will be used later. 
\begin{proposition}\label{prop:annulus}
	Assume that $(A_1)-(A_4)$ are satisfied. 
	Then $\tilde{\mu}_+(s)$ and $\tilde{\mu}_-(s)$ are continuous and strictly decreasing for sufficiently small $s \geqslant 0$. 
\end{proposition}

\medskip
The last (but not least) aim of the present article is the investigation of symmetry properties of least energy \textit{nodal} solutions to problem \eqref{D} via the results stated above. 
By a nodal (or, equivalently, sign-changing) solution of \eqref{D} we mean a weak solution $u$ such that $u^\pm := \max\{\pm u, 0\} \not\equiv 0$ in $\Omega$. 
By definition, a nodal set of $u$ is a set $Z = \overline{\{x \in \Omega: u(x) = 0\}}$, and any connected component of $\Omega \setminus Z$ is a nodal domain of $u$. 

Consider the nodal Nehari set
\begin{equation}\label{def:nodalNehari}
\mathcal{M} := \{u \in \wolp(\Omega):\, u^+ \in \mathcal{N}(\Omega), ~ -u^- \in \mathcal{N}(\Omega) \}.
\end{equation}
Evidently, $\mathcal{M}$ contains all nodal solutions of \eqref{D}. 
Moreover, in Appendix~\ref{section:appendix} we discuss that a least energy nodal solution of \eqref{D} can be found as a minimizer of the problem
\begin{equation}\label{nu}
\nu = \min_{u \in \mathcal{M}} E[u].
\end{equation}

Let $\Omega$ be a bounded radial domain in $\mathbb{R}^N$, that is, $\Omega$ is a ball or concentric annulus. 
The study of symmetric properties of least energy nodal solutions to~\eqref{D} in such $\Omega$ was initiated in \cite{BartschWethWillem}, where it was shown that in the linear case $p=2$ any minimizer of $\nu$ is a foliated Schwartz symmetric function with precisely two nodal domains. 
Here we consider the following question:
\begin{center}
	\textit{Is it true that any least energy nodal solution of \eqref{D} in a bounded radial $\Omega$ is nonradial?}
\end{center}
This question was first stated and answered affirmatively in \cite{aftalion} for the linear case $p=2$. The authors obtained the lower estimate $N+1$ on the Morse index of radial nodal solutions to \eqref{D} and used the fact that the Morse index of any least energy nodal solution of \eqref{D} is exactly $2$ (see \cite{bartschweth}). Note that the assumption $E \in C^2(\accentset{\circ}{W}_2^{1}(\Omega))$ is essential for the arguments of \cite{aftalion} and \cite{bartschweth}. 
Later, under weaker assumptions on $f$ which allow $E$ to be only $C^1(\accentset{\circ}{W}_2^{1}(\Omega))$, the nonradiality was proved in \cite{bartsch} by performing the idea of \cite{aftalion} in terms of a ``generalized'' Morse index. 
Nevertheless, necessity to work with the linearized problem associated with \eqref{D} at sign-changing solutions makes a generalization of the methods of \cite{aftalion} and \cite{bartsch} to the case $p>1$ nonobvious. (See \cite{castorina} about the linearization of the $p$-Laplacian).
We also refer to \cite{BartschWethWillem,T,CC} for some partial results on the nonradiality problem.

Here we give the affirmative answer on the above question on the nonradiality in the general case $p>1$ using different arguments based on shape optimization techniques. 
\begin{thm}\label{thm:3} 
	Let $\Omega$ be a ball or annulus and let $(A_1)-(A_4)$ be satisfied. 
	Then any minimizer of $\nu$ is nonradial and has precisely two nodal domains. 
\end{thm}

The fact that any minimizer of $\nu$ has exactly two nodal domains can be easily obtained by generalizing arguments from \cite[p.~1051]{cosio} or, equivalently, \cite[p.~6]{bartschweth}.
To prove the nonradiality, we argue by contradiction and apply Proposition~\ref{prop:annulus} to the least critical levels on eccentric annuli generated by small shifts of the nodal set of a radial nodal solution. Then, the union of least energy constant-sign solutions on modified in such a way nodal domains defines a function from $\mathcal{M}$ whose energy is strictly smaller than $\nu$. 

To the best of our knowledge, the idea to use shape optimization techniques for studying properties of nodal solutions was firstly performed in \cite{anoopdrabeksasi}, where it was proved that any second eigenfunction of the $p$-Laplacian on a ball cannot be radial. See also \cite{anoopbobsasi} about a development of this result.

It is worth mentioning that our approach has an intrinsic similarity with the methods of \cite{aftalion} and \cite{bartsch}. 
Consider a radial nodal solution $u$ of \eqref{D} with $k$ nodal domains $D_1, \dots, D_k$. Its nodal set is the union of $k-1$ concentric spheres $S_1, \dots, S_{k-1}$ inside $\Omega$. 
According to Proposition \ref{prop:annulus}, for any fixed $i \in \{1,\dots,k-1\}$ and $j \in \{1,\dots,N\}$, small shifts of $S_i$ along the coordinate axis $e_j$ generate a family of functions along which the energy functional $E$ strictly decreases. Thus, we have $(k-1)N$ such families generated by shifts. 
Moreover, considering the scaling $\alpha u|_{D_i}$, $\alpha>0$, for each of $k$ nodal components $D_1, \dots, D_k$, we produce $k$ additional families of functions with strictly decreasing energy, see Lemma \ref{lem:superlin}. Therefore, in total, we have $(k-1)N+k$ such families (compare with \cite[Theorem~2.2]{BartschWethWillem}). 
Without rigorous justification, we mention that this number can be seen as a weak variant of the Morse index of a radial nodal solution of \eqref{D} with $k$ nodal domains.

\medskip
The rest of the article is organized as follows. 
In Section~\ref{sec:derivative}, we study the dependence of $\mu_\pm(\Omega_t)$ on $t$ and prove Theorems~\ref{thm:1} and \ref{thm:1.5}.
Section~\ref{sec:optimization} is devoted to the study of the shape optimization problem for annular domains and contains the proof of Theorem~\ref{thm:2}.
In Section~\ref{sec:nonradial}, we prove the nonradiality of least energy nodal solutions to \eqref{D} stated in Theorem~\ref{thm:3}. 
Appendix~\ref{section:appendix} contains auxiliary results.

\section{Domain perturbations for least nontrivial critical levels}\label{sec:derivative}
For the proof of Theorems~\ref{thm:1} and \ref{thm:1.5} we need to prepare several auxiliary facts. 
Recall that $\Omega_t = \Phi_t(\Omega)$ is the deformation of $\Omega$, where the diffeomorphism $\Phi_t$ is given by \eqref{Phi}:
$$
\Phi_t(x) = x + t R(x), \quad R \in C^1(\mathbb{R}^N, \mathbb{R}^N), \quad |t| < \delta,
$$
and $\delta > 0$ is sufficiently small. 

Noting that any weak solution of \eqref{D} in $\Omega$ belongs to $C^{1,\beta}(\overline{\Omega})$ (see Remark~\ref{rem:regularity}), we state the following partial case of the generalized Pohozaev identity (see \cite[Lemma~2, p.~323]{degiovanni} with $\mathcal{L}(x,s,\xi) = \frac{1}{p}|\xi|^p - F(s)$).
\begin{proposition}\label{prop:Pohozaev}
	Let $u$ be any weak solution of \eqref{D} in $\Omega$. Then $u$ satisfies
	\begin{align}
	\notag
	\frac 1 p \int_{\Omega} |\nabla u|^p \, {\rm div} (R) \,dx - \int_{\Omega} |\nabla u|^{p-2} \left<\nabla u, \nabla u \cdot R'\right> dx &- \int_\Omega F(u) \,{\rm div} (R) \, dx \\
	\label{Pohozaev}
	&= -\frac{p-1}{p} \int_{\partial \Omega} \left|\frac{\partial u}{\partial n}\right|^p \left<R,n\right> d\sigma,
	\end{align}
	where $n$ is the outward unit normal vector to $\partial \Omega$, and $R'$ is the Jacobi matrix of $R$.
\end{proposition}

Fix now a nontrivial $u_0 \in \wolp(\Omega)$ and let $u_t \in \wolp(\Omega_t)$ be a function defined as $u_t(y) := u_0(\Phi_t^{-1}(y))$, $y \in \Omega_t$.
Although assertions of the following two lemmas can be deduced from general results  \cite{sokolowski,henrotpiere,delfour}, we give sketches of their proofs for the sake of completeness. 
\begin{lemma}\label{lem:diffPhi}
	Let $\phi \in C^1(-\delta, \delta)$. Then $\int_{\Omega_t} F(\phi(t) u_t(y)) \,dy$ is differentiable with respect to $t \in (-\delta, \delta)$ and
	\begin{equation}\label{eq:diffPhi}
	\left. \frac{\partial}{\partial t} \int_{\Omega_t} F(\phi(t) u_t(y)) \,dy \right |_{t = 0} = \phi'(0) \int_{\Omega} u_0 f (\phi(0) u_0) \,dx + \int_{\Omega} F(\phi(0) u_0) \,{\rm div} (R) \,dx. 
	\end{equation}
\end{lemma}
\begin{proof}
	Changing variables by the rule $y = \Phi_t(x)$ and noting that $dy = \,{\rm det} \left ( \frac{ d \Phi_t}{dx} \right ) dx$ for $|t| < \delta$, we obtain that 
	\begin{equation*}
	\int_{\Omega_t} F(\phi(t) u_t(y)) \,dy = \int_{\Omega} F(\phi(t) u_t(\Phi_t(x))) \,{\rm det} \left( \frac{ d \Phi_t}{dx} \right) dx = \int_{\Omega} F(\phi(t) u_0) \,{\rm det} \left(I + t R' \right) dx,
	\end{equation*}
	where $R'$ is the Jacobi matrix of $R$. 
	This implies the differentiability of $\int_{\Omega_t} F(\phi(t) u_t(y)) \,dy$ on $(-\delta, \delta)$. 
	On the other hand, from Jacobi's formula we know that
	\begin{equation}\label{prop:diffPhi2}
	\left.\frac{\partial}{\partial t} \,{\rm det} \left(I + t R' \right) \right |_{t = 0} = {\rm Tr} (R') = {\rm div} (R).
	\end{equation}
	Thus, differentiating $\int_{\Omega_t} F(\phi(t) u_t(y)) \,dy$ by $t$ at zero, we derive \eqref{eq:diffPhi}.
\end{proof}

By the same arguments as in the proof of Lemma~\ref{lem:diffPhi} we get the following fact.
\begin{cor}\label{cor:difff}
	$\int_{\Omega_t} u_t(y) f(\alpha u_t(y)) \,dy$ is differentiable with respect to $t \in (-\delta, \delta)$ for any $\alpha \in \mathbb{R}$.
\end{cor}

\begin{lemma}\label{lem:diffnabla}
	Let $\phi \in C^1(-\delta, \delta)$. Then $\int_{\Omega_t} |\nabla (\phi(t) u_t(y))|^p \,dy$ is differentiable with respect to $t \in (-\delta, \delta)$ and
	\begin{align}
	\notag
	\left. \frac{\partial}{\partial t} \int_{\Omega_t} |\nabla (\phi(t) u_t(y))|^p \,dy \right|_{t = 0}
	&= p |\phi(0)|^{p-2} \phi(0) \phi'(0) \int_{\Omega} |\nabla u_0|^p \,dx \\
	\label{eq:diffnabla}
	+ 
	|\phi(0)|^p \int_{\Omega} |\nabla u_0|^p \,{\rm div} (R) \, dx 
	&- 
	p |\phi(0)|^p \int_{\Omega} |\nabla u_0|^{p-2} \left<\nabla u_0, \nabla u_0 \cdot R'\right> dx.
	\end{align}
\end{lemma}
\begin{proof} 
	First, after the same change of variables as in the proof of Lemma~\ref{lem:diffPhi}, we obtain 
	\begin{equation}\label{eq:diffnabla1}
	\int_{\Omega_t} |\nabla u_t(y)|^p \,dy = \int_{\Omega_t} |\nabla u_0(\Phi_t^{-1}(y)) \cdot (\Phi_t^{-1}(y))'|^p \,dy = \int_{\Omega} |\nabla u_0 \cdot (\Phi_t')^{-1}|^p \,{\rm det} \left (I + t R' \right ) dx,
	\end{equation}
	where by $(\Phi^{-1}_t)'$ and $\Phi_t'$ we denote the corresponding Jacobi matrices and used the inversion property $(\Phi^{-1}_t(y))' = (\Phi_t'(x))^{-1}$. 
	Hence, \eqref{eq:diffnabla1} implies the differentiability of $\int_{\Omega_t} |\nabla u_t(y)|^p \,dy$ on $(-\delta, \delta)$.
	Note that the derivative of the inverse Jacobi matrix $(\Phi_t')^{-1}$ is given by 
	$$
	\left. \frac{\partial}{\partial t} (\Phi_t')^{-1} \right|_{t=0} = 
	- \left. (\Phi_t')^{-1} \cdot \frac{\partial \Phi_t'}{\partial t} \cdot (\Phi_t')^{-1} \right|_{t=0} 
	= - R'
	$$	
	since $\Phi'_t = I$ for $t=0$.
	Hence, differentiating \eqref{eq:diffnabla1} by $t$ at zero and taking into account \eqref{prop:diffPhi2}, we obtain
	\begin{equation*}
	\left . \frac{\partial}{\partial t} \int_{\Omega_t} |\nabla u_t(y)|^p \,dy \right |_{t = 0} = \int_{\Omega} |\nabla u_0|^p \,{\rm div} (R) \,dx - p \int_{\Omega} |\nabla u_0|^{p-2} \left<\nabla u_0, \nabla u_0 \cdot R'\right> \,dx.
	\end{equation*} 
	Finally, noting that $\nabla (\phi(t) u_t) = \phi(t) \nabla u_t$, we arrive at \eqref{eq:diffnabla}.
\end{proof}

Recall the definition \eqref{mu:Nehari} of the least nontrivial critical levels of $E$ in perturbed domains $\Omega_t$:
\begin{equation*}
\mu_+(\Omega_t) = \min_{v \in \mathcal{N}(\Omega_t), \, v \geqslant 0} E[v] 
\quad \text{and} \quad 
\mu_-(\Omega_t) = \min_{v \in \mathcal{N}(\Omega_t), \, v \leqslant 0} E[v].
\end{equation*}
From Appendix~\ref{section:appendix} (see Lemma~\ref{lem:existence} and Remark~\ref{rem:Omega_t}) we know that $\delta > 0$ can be chosen sufficiently small such that $\mu_+(\Omega_t)$ and $\mu_-(\Omega_t)$ possess minimizers for any $|t| < \delta$ which are constant-sign $C^{1,\beta}(\overline{\Omega_t})$-solutions of \eqref{D} in $\Omega_t$. 

Below in this section, we always denote by $v_0$ an arbitrary minimizer of $\mu_+(\Omega)$, that is, $v_0 \in \mathcal{N}(\Omega)$, $v_0 \geqslant 0$ in $\Omega$, and $E[v_0] = \mu_+(\Omega)$.
As above, consider the family of nonnegative functions  $v_t(y) := v_0(\Phi_t^{-1}(y))$, where $y \in \Omega_t$ and $|t| < \delta$.
We do not know that $v_t \in \mathcal{N}(\Omega_t)$. However, for each $|t| < \delta$ Lemma~\ref{lem:superlin} yields the existence of a unique $\alpha_t = \alpha(v_t)$ such that $\alpha_t > 0$ and  $\alpha_t v_t \in \mathcal{N}(\Omega_t)$. 

\begin{lemma}\label{lem:differ}
	$\alpha_t \in C^1(-\delta, \delta)$ and $\alpha_0 = 1$.
\end{lemma}
\begin{proof}
	Define $\Psi: (0,+\infty)\times(-\delta,\delta) \to \mathbb{R}$ by 
	$$
	\Psi(\alpha,t) = \alpha^{p-1} \int_{\Omega_t} |\nabla v_t|^p \, dy - \int_{\Omega_t} v_t f(\alpha v_t) \, dy.
	$$
	From Lemma~\ref{lem:diffnabla} and Corollary~\ref{cor:difff} we see that $\Psi(\alpha, \cdot)$ is differentiable on $(-\delta, \delta)$ for any $\alpha > 0$.
	On the other hand, we know that $v_0 > 0$ in $\Omega$ (see Remark~\ref{rem:maximum}) and hence $v_t > 0$ in $\Omega_t$. Thus, from $(A_1)$ it follows that $v_t f(\alpha v_t)$ is differentiable with respect to $\alpha > 0$ for each $x \in \Omega$ and $|t|<\delta$. Therefore, using $(A_2)$, we see that $\Psi(\cdot,t) \in C^1(0,+\infty)$ for any $|t|<\delta$.
	
	Since $v_0 \in \mathcal{N}(\Omega)$, we have $\Psi(1,0) = 0$. Moreover, in view of the first part of $(A_3)$ we have
	$$
	\Psi_\alpha'(1,0) = (p-1)\int_{\Omega} |\nabla v_0|^p \, dx - \int_{\Omega} v_0^2 f'(v_0) \, dx 
	= \int_{\Omega} v_0^2 \left( (p-1)\frac{f(v_0)}{v_0} - f'(v_0)\right) dx < 0.
	$$
	Hence, taking $\delta > 0$ smaller (if necessary), the implicit function theorem assures the existence of a differentiable function $\alpha_t: (-\delta,\delta) \to (0,+\infty)$ such that $\alpha_0 = 1$ and $\Psi(\alpha_t,t) = 0$ for all $t \in (-\delta,\delta)$, that is, $\alpha_t v_t \in \mathcal{N}(\Omega_t)$.
\end{proof}
\begin{remark}
	Consider any minimizer $w_0$ of $\mu_-(\Omega)$ and its deformation $w_t = w_0(\Phi_t^{-1}(y))$, $y \in \Omega_t$. Then the result of Lemma~\ref{lem:differ} remains valid for $\alpha_t  = \alpha(w_t)$ such that $\alpha_t w_t \in \mathcal{N}(\Omega_t)$.
\end{remark}

Now we are ready to prove Theorem~\ref{thm:1}. We give the proof of each statement separately.
\begin{proposition}\label{prop:differentiable}
	$E[\alpha_t v_t]$ is differentiable with respect to $t \in (-\delta, \delta)$ and
	\begin{equation}\label{mu:Neh0}
	\left.\frac{\partial E[\alpha_t v_t]}{\partial t}\right|_{t=0} = - \frac {p-1} p \int_{\partial \Omega} \left| \frac{\partial v_0}{\partial n} \right|^p \left<R, n\right> \,d\sigma.
	\end{equation}
\end{proposition}
\begin{proof}
First, $E[\alpha_t v_t]$ is differentiable due to Lemmas~\ref{lem:differ}, \ref{lem:diffPhi} and \ref{lem:diffnabla}. 
Moreover, applying equalities~\eqref{eq:diffPhi} and \eqref{eq:diffnabla} with $\phi(t) = \alpha_t$ and recalling that $\alpha_0 = 1$, we compute
\begin{align*}
\left.\frac{\partial E[\alpha_t v_t]}{\partial t}\right|_{t=0}
&= \alpha_0' \int_{\Omega} |\nabla v_0|^p \,dx + \frac 1 p \int_{\Omega} |\nabla v_0|^p \,{\rm div} (R) \,dx - \int_{\Omega} |\nabla v_0|^{p-2} \left<\nabla v_0, \nabla v_0 \cdot R' \right> dx \\ 
&- \alpha_0' \int_{\Omega} v_0 f (v_0) \,dx - \int_{\Omega} F(v_0) \,{\rm div} (R) \,dx.
\end{align*}
Since $v_0 \in \mathcal{N}(\Omega)$, the terms containing $\alpha_0'$ cancel out and we arrive at
\begin{equation*}
\left.\frac{\partial E[\alpha_t v_t]}{\partial t}\right|_{t=0} = \frac 1 p \int_{\Omega} |\nabla v_0|^p \,{\rm div} (R) \,dx - \int_{\Omega} |\nabla v_0|^{p-2} \left<\nabla v_0, \nabla v_0 \cdot R' \right> \,dx - \int_{\Omega} F(v_0) \,{\rm div} (R) \,dx.
\end{equation*}
Finally, applying the Pohozaev identity \eqref{Pohozaev}, we derive \eqref{mu:Neh0}.
\end{proof}

\begin{remark}
	Consider any minimizer $w_0$ of $\mu_-(\Omega)$ and its deformation $w_t = w_0(\Phi_t^{-1}(y))$, $y \in \Omega_t$. 
	Arguing as in Proposition~\ref{prop:differentiable}, we see that $E[\alpha(w_t) w_t]$ is also differentiable with respect to $t \in (-\delta, \delta)$ and satisfies the Hadamard-type formula~\eqref{mu:Neh2}.
\end{remark}

\begin{proposition}\label{proof:contin}
	$\mu_+(\Omega_t)$ and $\mu_-(\Omega_t)$ are continuous at $t=0$.
\end{proposition}
\begin{proof}
	We give the proof for $\mu_+(\Omega_t)$ only. The case of $\mu_-(\Omega_t)$ can be handled in much the same way. 
	Let us show that 
	$$
	\limsup\limits_{t \to 0} \mu_+(\Omega_t) \leqslant \mu_+(\Omega) \leqslant \liminf\limits_{t \to 0} \mu_+(\Omega_t).
	$$
	Suppose, by contradiction, that the first inequality does not hold. Consider a minimizer $v_0$ of $\mu_+(\Omega)$ and its deformation $v_t(y) := v_0(\Phi_t^{-1}(y))$, $y \in \Omega_t$. We know that  $E[\alpha(v_t) v_t] \geqslant \mu_+(\Omega_t)$. Moreover, $E[\alpha(v_t) v_t]$ is continuous with respect to $t \in (-\delta, \delta)$, see Proposition~\ref{prop:differentiable}. Therefore, we get 
	$$
	\limsup\limits_{t \to 0} E[\alpha(v_t) v_t] \geqslant \limsup\limits_{t \to 0} \mu_+(\Omega_t) > \mu_+(\Omega) = \limsup\limits_{t \to 0} E[\alpha(v_t) v_t],
	$$
	which is impossible.
	
	Suppose now, contrary to our claim, that $\mu_+(\Omega) > \liminf\limits_{t \to 0} \mu_+(\Omega_t)$. 
	Let $\{t_k\}_{k \in \mathbb{N}}$ be a sequence such that $\mu_+(\Omega) > \lim\limits_{k \to +\infty} \mu_+(\Omega_{t_k})$, and let $u_k \in \wolp(\Omega_{t_k})$ be a minimizer of $\mu_+(\Omega_{t_k})$, $k \in \mathbb{N}$. 
	We want to show that $\{u_k\}_{k \in \mathbb{N}}$ converges, up to a subsequence, to a minimizer of $\mu_+(\Omega)$. This will give us a contradiction. 
	Consider a smooth (nonempty) domain $\hat{\Omega} \subset \bigcap_{k \in \mathbb{N}} \Omega_{t_k}$. Extending each element of $\wolp(\hat{\Omega})$ outside of $\hat{\Omega}$ by zero, we see that $\mathcal{N}(\hat{\Omega}) \subset \mathcal{N}(\Omega_{t_k})$ for any $k \in \mathbb{N}$. 	
	Taking any $\xi \in C_0^\infty(\hat{\Omega})$, we apply Lemma~\ref{lem:superlin} to find an appropriate multiplier $\alpha > 0$ such that $\alpha \xi \in \mathcal{N}(\hat{\Omega})$, and hence $\mu_+(\Omega_{t_k}) \leqslant E[\alpha \xi]$ for any $k \in \mathbb{N}$. This implies that all $\|\nabla u_k\|_{L^p(\Omega_{t_k})}$ are uniformly bounded from above. Indeed, using $(A_2)$, $(A_4)$, and the first part of $(A_3)$, we get 
	\begin{align*}
	&0 < \int_{\Omega_{t_k}} F(u_k) \, dx \leqslant C_1 +  \frac{1}{\theta} \int_{\{x \in \Omega_{t_k}:\, u_k(x) > s_0\}} u_k f(u_k) \, dx 
	\leqslant C_1 +  \frac{1}{\theta} \int_{\Omega_{t_k}} u_k f(u_k) \, dx, 
	\end{align*}	
	where $C_1 > 0$ is chosen sufficiently large to be independent of $k$.
	Therefore, supposing that $\|\nabla u_k\|_{L^p(\Omega_{t_k})} \to +\infty$ as $k \to +\infty$ and recalling that $u_k \in \mathcal{N}(\Omega_{t_k})$, we obtain a contradiction:
	\begin{align*}
	\notag
	\mu_+(\Omega_{t_k})
	&= \frac{1}{p} \int_{\Omega_{t_k}} |\nabla u_k|^p \, dx - \int_{\Omega_{t_k}} F(u_k) \, dx\\
	&\geqslant \frac{1}{p} \int_{\Omega_{t_k}} |\nabla u_k|^p \, dx - \frac{1}{\theta} \int_{\Omega_{t_k}} u_k f(u_k) \, dx - C_1 =
	\left(\frac{1}{p}-\frac{1}{\theta}\right) \int_{\Omega_{t_k}} |\nabla u_k|^p \, dx - C_1 \to +\infty,
	\end{align*}
	since $\theta > p$.
		
	Consider now a bounded domain $\widetilde{\Omega} \supset \bigcup_{k \in \mathbb{N}} \Omega_{t_k}$.
	Extending each $u_k$ by zero outside of $\Omega_{t_k}$, we get $u_k \in \wolp(\widetilde{\Omega})$ and $\|\nabla u_k\|_{L^p(\widetilde{\Omega})} = \|\nabla u_k\|_{L^p(\Omega_{t_k})}$ for all $k \in \mathbb{N}$. 
	Therefore, the boundedness of $\{u_k\}_{k \in \mathbb{N}}$ in $\wolp(\widetilde{\Omega})$ implies the existence of $u \in \wolp(\widetilde{\Omega})$ such that, up to a subsequence, $u_k \to u$ weakly in $\wolp(\widetilde{\Omega})$ and strongly in $L^{q}(\widetilde{\Omega})$, $q \in (p, p^*)$. Moreover, since $u_k \geqslant 0$ in $\Omega_{t_k}$ for all $k \in \mathbb{N}$, we get $u \geqslant 0$ a.e.\ in $\widetilde{\Omega}$.
	
	In Remark~\ref{rem:Omega_t} below we show that the second part of $(A_3)$ yields $\limsup\limits_{s \to 0} \dfrac{f(s)}{|s|^{p-2}s} < \widetilde{C} < \lambda_p(\Omega_{t_k})$ for some $\widetilde{C} > 0$ and all $k$ large enough.
	Thus, due to the previous inequality and $(A_2)$, we can find $\mu \in (0, \widetilde{C})$ and $C_2>0$ such that $|f(s)| \leqslant \mu |s|^{p-1} + C_2 |s|^{q-1}$ for all $s \in \mathbb{R}$, where $q \in (p, p^*)$.
	Therefore, we get	
	$$
	\int_{\Omega_{t_k}} |\nabla u_k|^p \, dx = \int_{\Omega_{t_k}} u_k f(u_k) \, dx \leqslant \frac{\mu}{\lambda_p(\Omega_{t_k})} \int_{\Omega_{t_k}} |\nabla u_k|^p \, dx + C_3 \left(\int_{\Omega_{t_k}} |\nabla u_k|^p \, dx \right)^\frac{q}{p}
	$$
	for some $C_3 > 0$. 
	If we suppose that $\|\nabla u_k\|_{L^p(\Omega_{t_k})} \to 0$ as $k \to +\infty$, then for sufficiently large $k$ we obtain a contradiction since $\mu < \widehat{C} < \lambda_p(\Omega_{t_k})$ and $q>p$.
	Thus, there exists $C_4>0$ such that $\|\nabla u_k\|_{L^p(\Omega_{t_k})} > C_4$ for any $k$ large enough. This implies that $\int_{\Omega_{t_k}} u_k f(u_k) \, dx > C_4$ and hence $u \not\equiv 0$ a.e.\ in $\widetilde{\Omega}$. 
	
	Applying the fundamental lemma of calculus of variations, it is not hard to see that $u \equiv 0$ a.e.\ in $\mathbb{R}^N \setminus \Omega$. Since $\partial \Omega \in C^{2,\gamma}$, we conclude that $u \in \wolp(\Omega)$ (see, e.g., \cite[Theorem 5.29]{admasfournier}), and hence by the weak convergence we have 
	\begin{equation*}
	\|\nabla u\|_{L^p(\Omega)} = \|\nabla u\|_{L^p(\widetilde{\Omega})} 
	\leqslant 
	\liminf\limits_{k \to +\infty} \|\nabla u_k\|_{L^p(\widetilde{\Omega})} 
	= 
	\liminf\limits_{k \to +\infty} \|\nabla u_k\|_{L^p(\Omega_{t_k})}.
	\end{equation*} 
	Further, Lemma~\ref{lem:superlin} implies the existence of $\alpha(u) > 0$ such that $\alpha(u) u \in \mathcal{N}(\Omega)$. 
	Moreover, $E[\alpha u]$ achieves its unique maximum with respect to $\alpha>0$ at $\alpha(u)$. 
	On the other hand, since each $u_k \in \mathcal{N}(\Omega_{t_k})$, we deduce from Lemma~\ref{lem:superlin} that a unique point of maximum of $E[\alpha u_k]$ with respect to $\alpha>0$ is achieved at $\alpha = 1$. Therefore,
	\begin{align*}
	E[\alpha(u) u] 
	\leqslant \liminf\limits_{k \to +\infty} E[\alpha(u) u_k] 
	\leqslant \liminf\limits_{k \to +\infty} E[u_k] = \lim_{k \to +\infty} \mu_+(\Omega_{t_k})
	 < \mu_+(\Omega).
	\end{align*}
	Thus, recalling that $u \geqslant 0$ a.e.\ in $\Omega$, we get a contradiction to the definition of $\mu_+(\Omega)$, and hence $\mu_+(\Omega) \leqslant \liminf\limits_{t \to 0} \mu_+(\Omega_t)$. This completes the proof.
\end{proof}

\begin{remark}
	Proposition~\ref{proof:contin} implies that from any sequence of minimizers $u_k$ of $\mu_+(\Omega_{t_k})$ (or $\mu_-(\Omega_{t_k})$), $k \in \mathbb{N}$, one can extract a subsequence which converges strongly in $W^1_p(\mathbb{R}^N)$ to a minimizer of $\mu_+(\Omega)$ (or $\mu_-(\Omega)$). In view of possible nonuniqueness, the limit minimizer may depend on a sequence $\{t_k\}_{k \in \mathbb{N}}$.
\end{remark}

Theorem~\ref{thm:1.5} can be proved using the same arguments as for Theorem~\ref{thm:1} (even without normalization by $\alpha_t$ in view of homogeneity of the functional $J$ in \eqref{mu:Hom}).

\section{Optimization problem in annuli}\label{sec:optimization}
In this section we prove Theorem~\ref{thm:2}. 
Let us fix $R_1 > R_0 > 0$ and $s \in [0, R_1 - R_0)$. 
Consider problem \eqref{D} in the open spherical annulus $B_{R_1}(0) \setminus \overline{B_{R_0}(se_1)}$:
\begin{equation}\label{Da}
\left\{
\begin{aligned}
-\Delta_p u &= f(u) 
&&\text{in } B_{R_1}(0) \setminus \overline{B_{R_0}(se_1)}, \\
u &= 0  &&\text{on } \partial B_{R_1}(0) \text{ and } \partial B_{R_0}(se_1).
\end{aligned}
\right.
\end{equation}
Recall the notation $\tilde{\mu}_\pm(s) = \mu_\pm(B_{R_1}(0) \setminus \overline{B_{R_0}(se_1)})$ for the least nontrivial critical levels defined by \eqref{mu:Nehari} and consider a diffeomorphism $\Phi_t(x) = x + t R(x)$, $|t|<\delta$, with the vector field $R(x) = \varrho(x) e_1$, where $\varrho$ is a smooth function equal to zero in a neighborhood of $\partial B_{R_1}(0)$ and equal to one in a neighborhood of $\partial B_{R_0}(se_1)$. It is not hard to see that
$$
\Phi_t(B_{R_1}(0) \setminus \overline{B_{R_0}(se_1)}) = B_{R_1}(0) \setminus \overline{B_{R_0}((s+t)e_1)}.
$$
Therefore $\mu_\pm(\Phi_t(B_{R_1}(0) \setminus \overline{B_{R_0}(se_1)})) = \tilde{\mu}_\pm(s+t)$. 
This fact, together with Proposition~\ref{proof:contin}, implies the first part of Theorem~\ref{thm:2}.
\begin{lemma}\label{lem:contin}
	Let $(A_1)-(A_4)$ be satisfied. Then $\tilde{\mu}_+(s)$ and $\tilde{\mu}_-(s)$ are continuous for sufficiently small $s \geqslant 0$.
	If moreover $(A_3^*)$ holds, then $\tilde{\mu}_+(s)$ and $\tilde{\mu}_-(s)$ are continuous on $[0, R_1-R_0)$. 
\end{lemma}
Recall that imposing $(A_3^*)$ we can find minimizers of $\tilde{\mu}_\pm(s)$ for each $s \in [0, R_1-R_0)$, see the discussion in Section~\ref{sec:intro}. 
Without $(A_3^*)$ we can guarantee the existence of minimizers only for sufficiently small $s \geqslant 0$.

For simplicity of exposition we will give the proof of the second part of Theorem~\ref{thm:2} for $\tilde{\mu}_+(s)$ only. The case of $\tilde{\mu}_-(s)$ can be proved along the same lines. 
We will always assume that $(A_1)-(A_4)$ and $(A_3^*)$ are satisfied.

Let $v$ be an \textit{arbitrary} minimizer of $\tilde{\mu}_+(s)$, that is, $v$ is a least energy positive solution of \eqref{Da}. Recall that $v \in C^{1,\beta}(\overline{\Omega})$ and satisfies the Hopf maximum principle (see Remarks~\ref{rem:regularity} and \ref{rem:maximum} below). 
Defining $v_t(y) := v(\Phi_t^{-1}(y))$, $y \in B_{R_1}(0) \setminus \overline{B_{R_0}((s+t)e_1)}$, we have $\tilde{\mu}_+(s+t) \leqslant E[\alpha(v_t) v_t]$, where $\alpha(v_t)$ is given by Lemma~\ref{lem:superlin}.
Hence, noting that $\alpha(v) = 1$, from Theorem~\ref{thm:1} we obtain
\begin{align}
	\notag
	D&\tilde{\mu}_+(s) 
	:= \limsup_{t \to 0+} \frac{\tilde{\mu}_+(s+t) - \tilde{\mu}_+(s)}{t} \\
	&\leqslant
	\label{strange1}
	\limsup_{t \to 0+} \frac{E[\alpha(v_t) v_t] - E[\alpha(v) v]}{t} = \left.\frac{\partial E[\alpha(v_t) v_t]}{\partial t}\right|_{t=0} =  -\frac{p-1}{p} \int_{\partial B_{R_0}(se_1)} \left|\frac{\partial v}{\partial n} \right|^p n_1 \,d\sigma,
\end{align}
where $n_1 = n_1(x)$ is the first component of the outward unit normal vector $n$ to $\partial B_{R_0}(se_1)$.

Our main aim is to prove that $D\tilde{\mu}_+(s) < 0$ for all $s \in (0,R_1-R_0)$. 
In combination with the continuity of $\tilde{\mu}_+(s)$ (see Lemma~\ref{lem:contin}), this will immediately imply the desired strict monotonicity of $\tilde{\mu}_+(s)$ on $[0,R_1-R_0)$.

For the fixed $s \in [0, R_1-R_0)$ we write $\Omega := B_{R_1}(0) \setminus \overline{B_{R_0}(s e_1)}$, for simplicity. 
Denote by $H_{a}$ a hyperplane passing through the point $se_1$ (center of the inner ball) perpendicularly to a vector $a \neq 0$ which satisfies $\left<a, e_1\right>  \geqslant 0$.
Let $\rho_a: \mathbb{R}^N \to \mathbb{R}^N$ be a map which reflects a point $x \in \mathbb{R}^N$ with respect to $H_a$, and 
$\Sigma_a := \{x \in \mathbb{R}^N:\, \left<a, x-se_1\right> > 0\}$.
Note that under the assumption on $a$ we have $\rho_a(\Omega \cap \Sigma_a) \subseteq \{x \in \Omega:\, \left<a, x-se_1\right> < 0\}$.

First we prove the following fact.
\begin{lemma}\label{lem:mu<0}
	$D\tilde{\mu}_+(s) \leqslant 0$ for all $s \in [0,R_1-R_0)$. 
	Moreover, if $D\tilde{\mu}_+(s) = 0$ for some $s \in [0,R_1-R_0)$, then for any minimizer $v$ of $\tilde{\mu}_+(s)$ there exists $\varepsilon_0 > 0$ such that $v(x) = v(\rho_{e_1}(x))$ for all $x \in \partial B_{R_0+\varepsilon}(se_1)$ and $\varepsilon \in (0, \varepsilon_0)$. 
\end{lemma}
\begin{proof}
Let $v$ be a minimizer of $\tilde{\mu}_+(s)$ for some $s \in [0,R_1-R_0)$. Extend $v$ by zero outside of $\Omega$ and consider the following function:
\begin{equation}\label{polarization_v}
 V(x) = \begin{cases}
         \min(v(x),v(\rho_{e_1}(x))), &x \in \Sigma_{e_1},\\
         \max(v(x),v(\rho_{e_1}(x))), &x \in \mathbb{R}^N \setminus \Sigma_{e_1}.
        \end{cases}
\end{equation}
The function $V$ is the \textit{polarization} of $v$ with respect to $H_{e_1}$, see, e.g., \cite{brocksol,BartschWethWillem}. 
It is known that $V \in \wolp(\Omega)$, $V \geqslant 0$ in $\Omega$, and 
\begin{equation*}
 \int_{\Omega} |\nabla V|^p \,dx = \int_{\Omega} |\nabla v|^p \,dx,
 \quad
 \int_{\Omega} V f(V) \,dx = \int_{\Omega} v f(v) \,dx,
 \quad 
 \int_{\Omega} F(V) \,dx = \int_{\Omega} F(v) \,dx,
\end{equation*}
see \cite[Corollary~5.1]{brocksol} and \cite[Lemma~2.2]{BartschWethWillem}. 
In particular, $V \in \mathcal{N}(\Omega)$ and $E[V] = E[v]$, that is, $V$ is also a minimizer of $\tilde{\mu}_+(s)$. Since \eqref{strange1} holds for an arbitrary minimizer of $\tilde{\mu}_+(s)$, we arrive at 
\begin{equation*}
D\tilde{\mu}_+(s) \leqslant - \frac {p-1} p \int_{\partial B_{R_0}(se_1)} \left| \frac{\partial V}{\partial n} \right|^p n_1 \,d\sigma.
\end{equation*}
Now, since $V(x) = V(\rho_{e_1}(x)) = 0$ for all $x \in \partial B_{R_0}(se_1)$, and 
$V(x) \leqslant V(\rho_{e_1}(x))$ for all $x \in \Sigma_{e_1}$, we get 
\begin{equation}\label{domain_derivative2}
\frac{\partial V}{\partial n} (\rho_{e_1}(x)) \leqslant \frac{\partial V}{\partial n}(x) < 0 
\quad \text{for all } 
x \in \partial B_{R_0}(se_1) \cap \Sigma_{e_1}.
\end{equation}
Moreover, noting that $n_1(x) = - n_1(\rho_{e_1}(x))$ and $n_1(x) < 0$ for all $x \in \partial B_{R_0}(se_1) \cap \Sigma_{e_1}$, we get
\begin{equation}\label{domain_derivative1}
D\tilde{\mu}_+(s) \leqslant - \frac {p-1} p \int_{\partial B_{R_0}(se_1) \cap \Sigma_{e_1}} \left( \left| \frac{\partial V}{\partial n}(x) \right|^p - \left| \frac{\partial V}{\partial n}(\rho_{e_1}(x)) \right|^p \right) n_1(x) \,d\sigma \leqslant 0.
\end{equation}
This is the desired conclusion. 

Let us prove the second part of the lemma. Suppose that $D\tilde{\mu}_+(s) = 0$ for some $s \in [0,R_1-R_0)$. Polarizing any minimizer $v$ of $\tilde{\mu}_+(s)$ as above, we conclude from \eqref{domain_derivative2} and \eqref{domain_derivative1} that 
\begin{equation}\label{eq:assumvv}
\frac{\partial V}{\partial n} (\rho_{e_1}(x)) = \frac{\partial V}{\partial n}(x) < 0 
\quad \text{for all } 
x \in \partial B_{R_0}(se_1) \cap \Sigma_{e_1}.
\end{equation}
Define a function $w(x) = V(\rho_{e_1}(x)) - V(x)$. 
By the properties of $V$ we have $w \geqslant 0$ in $\Omega \cap \Sigma_{e_1}$ and $w=0$ on $\partial B_{R_0}(se_1) \cap \Sigma_{e_1}$.
Moreover, since $\frac{\partial V}{\partial n} < 0$ on $\partial B_{R_0}(se_1)$, we can find $\varepsilon_0 > 0$ such that for any $\varepsilon \in (0, \varepsilon_0)$ there exists $\eta > 0$ such that $|\nabla V| > \eta$ in $B_{R_0+\varepsilon}(se_1) \setminus B_{R_0}(se_1)$. 
Therefore, using Remark~\ref{rem:reg2}, we can linearize the difference $\Delta_p V(\rho_{e_1}(\cdot)) - \Delta_p V$ in $(B_{R_0 + \varepsilon}(se_1) \setminus B_{R_0}(se_1)) \cap \Sigma_{e_1}$ as in the proof of \cite[Proposition~5.1]{fleckingertakac} on p.~1239 and obtain that $w$ satisfies pointwise the following linear elliptic inequality in this set:
\begin{equation}
\label{eq:oper0}
\sum_{i,j=1}^{N} a_{ij}(x) \frac{\partial^2 w}{\partial x_i \partial x_j} + \sum_{i=1}^{N} b_i(x) \frac{\partial w}{\partial x_i} = f(V(\rho_{e_1}(x))) - f(V(x)) \geqslant 0.
\end{equation}
(The inequality in \eqref{eq:oper0} follows from the monotonicity of $f$ on $\mathbb{R}$, see $(A_3)$.)
The matrix $\{a_{ij}\}_{i,j=1}^N$ is symmetric and there exist $C_1 = C_1(\eta) > 0$ and $C_2 = C_2(\eta) > 0$ such that 
$$
C_1 |\xi|^2 \leqslant \sum_{i,j=1}^{N} a_{ij}(x) \xi_i \xi_j \leqslant C_2 |\xi|^2
$$
for any $x \in (B_{R_0 + \varepsilon}(se_1) \setminus B_{R_0}(se_1)) \cap \Sigma_{e_1}$ and $\xi \in \mathbb{R}^N \setminus \{0\}$. 
That is, the elliptic operator on the left-hand side of \eqref{eq:oper0} is uniformly elliptic. Moreover, each $b_i \in L^\infty((B_{R_0 + \varepsilon}(se_1) \setminus B_{R_0}(se_1)) \cap \Sigma_{e_1})$. 
Hence, the classical strong maximum principle \cite[Theorem~3.5, p.~35]{giltrud} implies that
either $V(x) = V(\rho_{e_1}(x))$ for any $x \in (B_{R_0 + \varepsilon}(se_1) \setminus B_{R_0}(se_1)) \cap \Sigma_{e_1}$, or $V(x) < V(\rho_{e_1}(x))$ in the same set and 
\begin{equation}\label{eq:strict1}
\frac{\partial V}{\partial n}(\rho_{e_1}(x)) <
\frac{\partial V}{\partial n}(x) 
< 0 
\quad \text{for all } x \in \partial B_{R_0}(se_1) \cap \Sigma_{e_1}
\end{equation}
by \cite[Lemma~3.4, p.~34]{giltrud}. 
However, in view of \eqref{eq:assumvv} only the first case can occur. 
Thus, from the definition of the polarization we obtain the desired fact: $v(x) = v(\rho_{e_1}(x))$ for any $x \in B_{R_0 + \varepsilon}(se_1) \setminus B_{R_0}(se_1)$ and $\varepsilon \in (0, \varepsilon_0)$.
\end{proof}

\begin{remark}\label{rem:noscp}
	In the second part of the proof of Lemma~\ref{lem:mu<0}, $V$ and $V(\rho_{e_1}(\cdot))$ satisfy $V(x) < V(\rho_{e_1}(x))$ for all $x \in \partial B_{R_1}(0) \cap \Sigma_{e_1}$ whenever $s \in (0, R_1-R_0)$. Therefore, in the case $p=2$ the classical strong maximum (comparison) principle implies that $V(x) < V(\rho_{e_1}(x))$ for all $x \in \Omega \cap \Sigma_{e_1}$ and \eqref{eq:strict1} holds. 
	This yields $D\tilde{\mu}_+(s) < 0$ for any $s \in (0, R_1-R_0)$. 
	However, the lack of strong comparison principles in the general case $p>1$ does not allow to conclude directly that $D\tilde{\mu}_+(s) < 0$. 
	(We refer to \cite{damascelli,sciunzi} for versions of the strong comparison principle under the restriction $p > \frac{2N+2}{N+2}$.) 
	On the other hand, the arguments which we use below do not require any global strong comparison result and rely mainly on the (local) strong comparison principle near the boundary of $\Omega$, where the $p$-Laplacian is neither degenerate nor singular thanks to the Hopf maximum principle. 
\end{remark}

Now we show the following result on existence of \textit{axially symmetric} minimizers of $\tilde{\mu}_+(s)$.
\begin{lemma}\label{lem:sphericallysym}
	For any $s \in [0, R_1-R_0)$ there exists a minimizer of $\tilde{\mu}_+(s)$ invariant under rotations around axis $e_1$.
\end{lemma}
\begin{proof}
	Let $v$ be an arbitrary minimizer of $\tilde{\mu}_+(s)$ for some $s \in [0, R_1-R_0)$. 
	Recall that $v \in C^{1,\beta}(\overline{\Omega})$ and $v > 0$ in $\Omega$, see Remarks~\ref{rem:regularity} and \ref{rem:maximum} below. 
	To prove the assertion we apply the spherical symmetrization for $v$ with respect to $-e_1$. Namely, for a set $A \subset \mathbb{R}^N$ its spherical symmetrization around $-e_1$ is a set  $A^*$ defined such that for any $r > 0$, $A^* \cap \partial B_r(0)$ is a spherical cap of $\partial B_r(0)$ with the pole $-re_1$ and $\text{meas}(A^* \cap \partial B_r(0)) = \text{meas}(A \cap \partial B_r(0))$, see, e.g., \cite{kawohl,BartschWethWillem}.
	Then, the spherical symmetrization of $v$ around $-e_1$ is a function $v^*: \mathbb{R}^N \to \mathbb{R}$ defined as 
	$$
	\{x \in \Omega^*:\, v^*(x) \geqslant t\} = \{x \in \Omega:\, v(x) \geqslant t\}^* \quad \text{for all} \quad t \geqslant 0.
	$$
	By construction, $v^*$ is invariant under rotations around $e_1$. 
	Due to the symmetry of $\Omega = B_{R_1}(0) \setminus \overline{B_{R_0}(se_1)}$, we have $\Omega^* = \Omega$. Thus, $v^* \in \wolp(\Omega)$, as follows from \cite[property (L), p.~20]{kawohl}. Moreover, \cite[properties (G1), p.~26, and (C), p.~22]{kawohl} imply that
	$$
	\int_\Omega |\nabla v^*|^p \, dx \leqslant \int_\Omega |\nabla v|^p \, dx, 
	\quad 
	\int_\Omega v^* f(v^*) \, dx = \int_\Omega vf(v)\, dx,
	\quad 
	\int_\Omega F(v^*) \, dx = \int_\Omega F(v)\, dx.
	$$
	If in the first expression equality holds, then $v^* \in \mathcal{N}(\Omega)$, $v^* \geqslant 0$ and $E[v^*] = E[v]$, that is, $v^*$ is a minimizer of $\tilde{\mu}_+(s)$ with the desired properties. 
	Else, we get a contradiction. Indeed, using Lemma~\ref{lem:superlin}, we can find $\alpha^* \in (0,1)$ such that $\alpha^* v^* \in \mathcal{N}(\Omega)$. However,
	\begin{multline*}
	 \tilde{\mu}_+(s) \leqslant E[\alpha^* v^*] = E[\alpha^* v^*] - \frac{1}{p} E'[\alpha^* v^*](\alpha^* v^*) = \int_{\Omega} \left( \frac{1}{p} \alpha^* v^* f(\alpha^* v^*) - F(\alpha^* v^*) \right) dx < \\
         < \int_{\Omega} \left( \frac{1}{p}  v^* f(v^*) - F(v^*) \right) dx = \int_{\Omega} \left( \frac{1}{p}  v f(v) - F(v) \right) dx = E[v] - \frac{1}{p} E'[v]v = E[v] = \tilde{\mu}_+(s),
	\end{multline*}
	where the strict inequality follows from the first part of $(A_3)$, which is impossible. 
\end{proof}

\begin{remark}\label{rem:nondiff}
	In \cite{coffman} it was proved that in the case $p=2$ and $N=2$, $\tilde{\mu}_+(0)$ possesses a nonradial minimizer $v$ if the annulus is sufficiently thin (see also \cite{Nazar2, Nazar, Kolon} and references therein for the development of this result for $p>1$ and $N \geqslant 2$). 
	Using Lemma~\ref{lem:sphericallysym}, we can assume that $v$ is axially symmetric with respect to $e_1$. Moreover, the spherical symmetrization implies, in fact, that $v$ is polarized with respect to $H_{e_1}$, that is, $v(x) \leqslant v(\rho_{e_1}(x))$ for all $x \in \Omega \cap \Sigma_{e_1}$. 
	However, since $v$ is nonradial, the classical strong maximum principle implies that $v(x) < v(\rho_{e_1}(x))$ in this domain, which yields $D\tilde{\mu}_+(0) < 0$, see~\eqref{domain_derivative1}. 
	This fact contradicts the possible differentiability of $\tilde{\mu}_+(s)$ at $s=0$. Indeed, if $\tilde{\mu}_+(0)$ is differentiable, then $D\tilde{\mu}_+(0) = (\tilde{\mu}_+(0))'_s$ and we must have $(\tilde{\mu}_+(0))'_s = 0$, due to the symmetry of $\Omega$. 
\end{remark}

The following lemma provides the main ingredient for the proof of $D\tilde{\mu}_+(s) < 0$ for $s \in (0, R_1-R_0)$. (See \cite[Theorem~3.8]{anoopbobsasi} about the analogous properties for the first eigenvalue $\lambda_p(s)$.)
\begin{lemma}\label{lem:rad1}
	Let $D\tilde{\mu}_+(s) = 0$ for some $s \in [0,R_1-R_0)$.
	Then for any axially symmetric (with respect to $e_1$) minimizer $v$ of $\tilde{\mu}_+(s)$ there is a ball $B_{r_0}(se_1)$ with $r_0 \in (R_0, R_1-s)$ such that $v$ is radial in the annulus $B_{r_0}(se_1) \setminus B_{R_0}(se_1)$. Moreover, $|\nabla v| = 0$ on $\partial B_{r_0}(se_1)$ and $v \in C^2(B_{r_0}(se_1) \setminus B_{R_0}(se_1))$.
\end{lemma}
\begin{proof}
 	Let $D\tilde{\mu}_+(s) = 0$ for some $s \in [0,R_1-R_0)$ and let $v$ be a minimizer of $\tilde{\mu}_+(s)$ which is invariant under rotations around $e_1$ (see Lemma~\ref{lem:sphericallysym}).	
 	Due to the Hopf maximum principle (see Remark~\ref{rem:maximum} below) we can find $\varepsilon_0 > 0$ (as in Lemma~\ref{lem:mu<0}) such that for any $\varepsilon \in (0, \varepsilon_0)$ there exists $\eta > 0$ such that $|\nabla v| > \eta$ in $B_{R_0+\varepsilon}(se_1) \setminus B_{R_0}(se_1)$.  
 	
 	Suppose, contrary to the radiality of $v$, that for some $\varepsilon \in (0, \varepsilon_0)$ there exist $\hat{x}, \hat{y} \in \partial B_{R_0+\varepsilon}(se_1)$ such that $v(\hat{x}) \neq v(\hat{y})$. 
 	Note that $\hat{y} \neq \rho_{e_i}(\hat{x})$ for $i=2,\dots,N$ since $v$ is axially symmetric with respect to $e_1$. 
 	Moreover, $\hat{y} \neq \rho_{e_1}(\hat{x})$, as it follows from Lemma~\ref{lem:mu<0}. 	
 	Let $\bar{x}, \bar{y} \in \partial B_{R_0 + \varepsilon}(se_1)$ lie on the opposite sides of the diameter of $B_{R_0 + \varepsilon}(se_1)$ which is collinear to $\hat{x}-\hat{y}$. Assume, without loss of generality, that $\bar{x}_1 \geqslant s$. 
 	Using the symmetries of $v$ given by Lemmas~\ref{lem:mu<0} and \ref{lem:sphericallysym}, we derive that 
 	\begin{align*}
 	\notag
 	v(\bar{x}_1, \dots, \bar{x}_N) &= v(2s-\bar{x}_1, \bar{x}_2, \dots, \bar{x}_N) \\
 	\label{eq:symmetryen}
 	&= v(2s-\bar{x}_1, -\bar{x}_2, \dots, \bar{x}_N) = \dots = 
 	v(2s-\bar{x}_1, \dots, -\bar{x}_N) \equiv v(\bar{y}_1, \dots, \bar{y}_N).
 	\end{align*}
 	Let us denote $c = \bar{x}-\bar{y}$ and consider the polarization of $v$ with respect to $H_c$:
 	\begin{equation}\label{polarization_c}
 	V_c(x) = \begin{cases}
 	\min(v(x),v(\rho_c(x))),	&x \in \Sigma_c,\\
 	\max(v(x),v(\rho_c(x))),	&x \in \mathbb{R}^N \setminus \Sigma_c.
 	\end{cases}
 	\end{equation}
 	Define a function $w(x) = V_c(\rho_c(x)) - V_c(x)$. We have $w \geqslant 0$ in $\Omega \cap \Sigma_c$ and $w=0$ on $\partial B_{R_0}(se_1) \cap \Sigma_c$. 
 	Using the linearization of the $p$-Laplacian as in the proof of Lemma~\ref{lem:mu<0}, we see that $w$ satisfies
 	\begin{equation*}\label{eq:oper}
 	\sum_{i,j=1}^{N} a_{ij}(x) \frac{\partial^2 w}{\partial x_i \partial x_j} + \sum_{i=1}^{N} b_i(x) \frac{\partial w}{\partial x_i} = f(V_c(\rho_c(x))) - f(V_c(x)) \geqslant 0
 	\end{equation*}
 	pointwise in $(B_{R_0 + \varepsilon}(se_1) \setminus \overline{B_{R_0}(se_1)}) \cap \Sigma_c$.
 	Hence, either $V_c(x) = V_c(\rho_c(x))$ for all $x \in (B_{R_0 + \varepsilon}(se_1) \setminus \overline{B_{R_0}(se_1)}) \cap \Sigma_c$, or $V_c(x) < V_c(\rho_c(x))$ for all $x \in (B_{R_0 + \varepsilon}(se_1) \setminus \overline{B_{R_0}(se_1)}) \cap \Sigma_c$.
 	However, we have simultaneously $V_c(\hat{x}) < V_c(\hat{y}) \equiv V_c(\rho_c(\hat{x}))$ and $V_c(\bar{x}) = V_c(\bar{y}) \equiv V_c(\rho_c(\bar{x}))$, which is impossible.
 	
 	Finally, considering the least $\varepsilon_0 > 0$ such that $|\nabla v(x)| = 0$ occurs for some $x \in \partial B_{R_0 + \varepsilon_0}(se_1)$, we obtain that $v$ is radial in the annulus $B_{R_0 + \varepsilon}(se_1) \setminus B_{R_0}(se_1)$ for any $\varepsilon \in (0, \varepsilon_0)$. Denoting $r_0 = R_0 + \varepsilon_0$ and referring to Remark~\ref{rem:reg2} below for the $C^2$-regularity, we finish the proof of the lemma.
 \end{proof}

Let us now outline how to prove symmetry results similar to Lemma~\ref{lem:rad1}, but in a neighborhood of the outer ball $B_{R_1}(0)$. 
Consider a diffeomorphism $\bar{\Phi}_t(x) = x + t \bar{R}(x)$, where the vector field $\bar{R}(x) = -\bar{\varrho}(x) e_1$ and $\bar{\varrho}$ is a smooth function equal to one in a neighborhood of $\partial B_{R_1}(0)$ and equal to zero in a neighborhood of $\partial B_{R_0}(se_1)$.  We see that
$$
\bar{\Phi}_t(B_{R_1}(0) \setminus \overline{B_{R_0}(se_1)}) = B_{R_1}(-te_1) \setminus \overline{B_{R_0}(se_1)}
$$
for all $|t|$ small enough. 
Taking into account the invariance of $\mu_+(\Omega)$ under translations of $\Omega$, we get 
$\mu_+(\bar{\Phi}_t(B_{R_1}(0) \setminus \overline{B_{R_0}(se_1)})) = \tilde{\mu}_+(s+t)$.
Therefore, similarly to \eqref{strange1}, we obtain the following upper estimate for $D\tilde{\mu}_+(s)$:
\begin{align}
\notag
D&\tilde{\mu}_+(s)
= \limsup_{t \to 0+} \frac{\tilde{\mu}_+(s+t) - \tilde{\mu}_+(s)}{t} \\
&\leqslant
\label{strange2}
\limsup_{t \to 0+} \frac{E[\alpha(v_t) v_t] - E[\alpha(v)v]}{t} = \left.\frac{\partial E[\alpha(v_t) v_t]}{\partial t} \right|_{t=0} = \frac {p-1} p \int_{\partial B_{R_1}(0)} \left | \frac{\partial v}{\partial n} \right |^p n_1 \,d\sigma.
\end{align}

Denote by $\bar{H}_{a}$ a hyperplane passing through the origin (center of the outer ball) perpendicularly to a vector $a \neq 0$ which satisfies $\left<a, e_1\right> \geqslant 0$.
Let $\bar{\rho}_a(x)$ be a reflection of $x \in \mathbb{R}^N$ with respect to $\bar{H}_a$, and $\bar{\Sigma}_a = \{x \in \mathbb{R}^N:\, \left<a,x\right> > 0\}$. Under the assumption on $a$, we have $\bar{\rho}_a(\Omega \cap \bar{\Sigma}_a) \subseteq \{x \in \Omega:\, \left<a,z\right> < 0\}$. 
Consider the corresponding polarization of a minimizer $v$ of $\tilde{\mu}_+(s)$:
\begin{equation*}\label{polarization_o}
	\bar{V}_a(x) = \begin{cases}
	\min(v(x),v(\bar{\rho}_a(x))),	&x \in \bar{\Sigma}_a,\\
	\max(v(x),v(\bar{\rho}_a(x))),	&x \in \mathbb{R}^N \setminus \bar{\Sigma}_a.
	\end{cases}
\end{equation*}
It is not hard to see that $\text{supp}\, \bar{V}_a = \overline{\Omega}$ and hence $\bar{V}_a \in \wolp(\Omega)$.

Arguing now along the same lines as in the proofs of Lemmas~\ref{lem:mu<0} and \ref{lem:rad1} with the use of \eqref{strange2} instead of \eqref{strange1}, polarizations $\bar{V}_{e_1}$ and $\bar{V}_c$ instead of \eqref{polarization_v} and \eqref{polarization_c}, respectively, and the linearization of the $p$-Laplacian in a neighborhood of $\partial B_{R_1}(0)$ instead of $\partial B_{R_0}(se_1)$, we obtain the following results. 
\begin{lemma}\label{lem:mu<01}
	Let $D\tilde{\mu}_+(s) = 0$ for some $s \in [0,R_1-R_0)$. Then for any minimizer $v$ of $\tilde{\mu}_+(s)$ there exists $\varepsilon_1 > 0$ such that $v(x) = v(\bar{\rho}_{e_1}(x))$ for all $x \in \partial B_{R_1-\varepsilon}(0)$ and $\varepsilon \in (0, \varepsilon_1)$. 
\end{lemma}
\begin{lemma}\label{lem:rad2}
	Let $D\tilde{\mu}_+(s) = 0$ for some $s \in [0,R_1-R_0)$.
	Then for any axially symmetric (with respect to $e_1$) minimizer $v$ of $\tilde{\mu}_+(s)$ there is a ball $B_{r_1}(0)$ with $r_1 \in (R_0+s, R_1)$ such that $v$ is radial in the annulus $B_{R_1}(0) \setminus B_{r_1}(0)$. Moreover, $|\nabla v| = 0$ on $\partial B_{r_1}(0)$ and $v \in C^2(\overline{B_{R_1}(0)} \setminus \overline{B_{r_1}(0}))$.
\end{lemma}

Now we are ready to prove the main result which implies strict monotonicity of $\tilde{\mu}_+(s)$ on $[0, R_1-R_0)$, that is, the second part of Theorem~\ref{thm:2}.
\begin{proposition}
	$D\tilde{\mu}_+(s) < 0$ for all $s \in (0, R_1-R_0)$.
\end{proposition}
\begin{proof}
Suppose, by contradiction, that $D\tilde{\mu}_+(s) = 0$ for some $s \in (0, R_1-R_0)$. Let $v$ be an axially symmetric minimizer of $\tilde{\mu}_+(s)$ given by Lemma~\ref{lem:sphericallysym}.
From Lemmas~\ref{lem:rad1} and \ref{lem:rad2} we know that there exist $r_0$ and $r_1$ such that 
$v \in C^2(\overline{B_{R_1}(0)} \setminus B_{r_1+\varepsilon}(0))$ and $v \in C^2(\overline{B_{r_0-\varepsilon}(se_1)} \setminus B_{R_0}(se_1))$ for any sufficiently small $\varepsilon > 0$, and hence $v$ satisfies \eqref{D} pointwise in the corresponding domains. 
Let us multiply \eqref{D} by $v$, integrate it over $B_{R_1}(0) \setminus \overline{B_{r_1+\varepsilon}(0)}$ and tend $\varepsilon \to 0$. We get
\begin{equation}\label{eq:neh_v1}
\int_{B_{R_1}(0) \setminus \overline{B_{r_1}(0)}} |\nabla v|^p \, dx - \int_{B_{R_1}(0) \setminus \overline{B_{r_1}(0)}} v f(v) \, dx = 0,
\end{equation}
since $v = 0$ on $\partial B_{R_1}(0)$ and $|\nabla v| \to 0$ on $\partial B_{r_1+\varepsilon}(0)$ as $\varepsilon \to 0$, due to Lemma~\ref{lem:rad2}. Analogously, we obtain
\begin{equation}\label{eq:neh_v3}
\int_{B_{r_0}(se_1) \setminus \overline{B_{R_0}(se_1)}} |\nabla v|^p \, dx - \int_{B_{r_0}(se_1) \setminus \overline{B_{R_0}(se_1)}} v f(v) \, dx = 0.
\end{equation}
Note that 
\begin{equation*}
\int_{\Omega} |\nabla v|^p \, dx 
= \int_{B_{R_1}(0) \setminus \overline{B_{r_1}(0)}} |\nabla v|^p \, dx + 
\int_{B_{r_1}(0) \setminus \overline{B_{r_0}(se_1)}} |\nabla v|^p \, dx + 
\int_{B_{r_0}(se_1) \setminus \overline{B_{R_0}(se_1)}} |\nabla v|^p \, dx,
\end{equation*}
and similar decompositions hold for $\int_{\Omega} v f(v) \, dx$ and $\int_{\Omega} F(v) \, dx$. 
Therefore, recalling that $v \in \mathcal{N}(\Omega)$ and using \eqref{eq:neh_v1} and \eqref{eq:neh_v3}, we derive also 
\begin{equation}\label{eq:neh_v2}
\int_{B_{r_1}(0) \setminus \overline{B_{r_0}(se_1)}} |\nabla v|^p \, dx - \int_{B_{r_1}(0) \setminus \overline{B_{r_0}(se_1)}} v f(v) \, dx = 0.
\end{equation}
In other words, $v$ satisfies the Nehari constraint over each of the domains 
$$
B_{R_1}(0) \setminus \overline{B_{r_1}(0)}, \quad B_{r_1}(0) \setminus \overline{B_{r_0}(se_1)}, \quad  B_{r_0}(se_1) \setminus \overline{B_{R_0}(se_1)}.
$$
Let us consider a function $w: \Omega \to \mathbb{R}$ defined by
$$
w(x) = 
\left\{
\begin{aligned}
&C_1 v (x)	&&x \in B_{R_1}(0) \setminus \overline{B_{r_1}(0)},\\
&C_2,		&&x \in B_{r_1}(0) \setminus \overline{B_{r_0}(se_1)},\\
&v(x), 		&&x \in B_{r_0}(se_1) \setminus \overline{B_{R_0}(se_1)},
\end{aligned}
\right.
$$
where constants $C_1, C_2 > 0$ are chosen such that $C_1 v|_{\partial B_{r_1}(0)} = C_2 = v|_{\partial B_{r_0}(se_1)}$. (Note that $v$ is constant on $\partial B_{r_1}(0)$ and $\partial B_{r_0}(se_1)$ due to Lemmas~\ref{lem:rad1} and \ref{lem:rad2}, respectively.)
Therefore, $w > 0$ in $\Omega$, $w \in C^{1}(\overline{\Omega})$, and there exists a unique $\alpha(w) > 0$ such that $\alpha(w) w \in \mathcal{N}(\Omega)$ and $E[\alpha w]$ achieves a global maximum with respect to $\alpha > 0$ at $\alpha(w)$, see Lemma~\ref{lem:superlin}.
Thus, we have
\begin{equation}\label{eq:mu<E}
\tilde{\mu}_+(s) \leqslant E[\alpha(w) w] = E_1[\alpha(w)C_1 v] + E_2[\alpha(w)C_2] + E_3[\alpha(w)v],
\end{equation}
where $E_1$ denotes a restriction of $E$ to the domain of integration $B_{R_1}(0) \setminus \overline{B_{r_1}(0)}$, etc.

Recalling \eqref{eq:neh_v1} and \eqref{eq:neh_v3}, Lemma~\ref{lem:superlin2} implies that $E_1[\alpha(w)C_1 v] \leqslant E_1[v]$ and $E_3[\alpha(w)v] \leqslant E_3[v]$. 
Moreover, since $v > 0$ in $\Omega$, from \eqref{eq:neh_v2} we get $E_2[v] > 0$ by Lemma~\ref{lem:superlin2}. 
However, 
\begin{align*}
E_2[\alpha(w)C_2] 
&= \int_{B_{r_1}(0) \setminus B_{r_0}(se_1)} |\nabla \alpha(w)C_2|^p \,dx - \int_{B_{r_1}(0) \setminus B_{r_0}(se_1)} F(\alpha(w)C_2) \, dx \\
&= - \int_{B_{r_1}(0) \setminus B_{r_0}(se_1)} F(\alpha(w)C_2) \, dx < 0.
\end{align*}
Thus, 
$$
E_1[\alpha(w)C_1 v] + E_2[\alpha(w)C_2] + E_3[\alpha(w)v] < E_1[v] + E_2[v] + E_3[v] = E[v] = \tilde{\mu}_+(s),
$$
and we get a contradiction with \eqref{eq:mu<E}.
\end{proof}

\section{Nonradiality of least energy nodal solutions}\label{sec:nonradial}
In this section we prove Theorem~\ref{thm:3}, that is, we show that any least energy  nodal solution of problem \eqref{D} in a ball or annulus is nonradial. 
First, we treat the case of a ball. Consider the problem 
\begin{equation}\label{Dn}
\left\{
\begin{aligned}
-\Delta_p u &= f(u) 
&&\text{in } B_R(0), \\
u &= 0  &&\text{on } \partial B_R(0),
\end{aligned}
\right.
\end{equation}	
where $B_R(0)$ is the open ball with some radius $R$ centered at the origin, and $f$ satisfies $(A_1)-(A_4)$.
Recall that any least energy nodal solution of \eqref{Dn} is a minimizer of
\begin{equation*}
\nu = \min_{u \in \mathcal{M}} E[u],
\end{equation*}
where $\mathcal{M}$ is the nodal Nehari set \eqref{def:nodalNehari}.

Suppose, by contradiction, that there exists a minimizer $u$ of $\nu$ which is radial. Hence, there exists $r \in (0,R)$ such that, without loss of generality, $u^+$ is a least energy positive solution of \eqref{D} in the annulus $B_R(0) \setminus \overline{B_r(0)}$ and $-u^-$ is a least energy negative solution of \eqref{D} in the ball $B_r(0)$. 
As in Section~\ref{sec:optimization}, let us perturb $B_R(0) \setminus \overline{B_r(0)}$ by shifting the inner ball in direction $e_1$. 
From Lemma~\ref{lem:existence} and Remark~\ref{rem:Omega_t} we know that $\mu_+(B_R(0) \setminus \overline{B_r(se_1)})$ possesses a minimizer $v_s$ for any $s \geqslant 0$ small enough, and 
\begin{equation}\label{eq:nonrad1}
\mu_+(B_R(0) \setminus \overline{B_r(se_1)}) < \mu_+(B_R(0) \setminus \overline{B_r(0)})
\end{equation}
by  Proposition~\ref{prop:annulus}.
Extending $v_s$ by zero outside of $B_R(0) \setminus \overline{B_r(se_1)}$, we get $v_s \in \mathcal{N}(B_R(0) \setminus \overline{B_r(se_1)}) \subset \mathcal{N}(B_R(0))$.
On the other hand, it is not hard to see that the translation $-u^-(\cdot - se_1) \in \mathcal{N}(B_R(0))$ and $E[-u^-(\cdot - se_1)] = E[-u^-]$.
Therefore, if we consider a function $U_s$ defined as $U_s(x) = v_s(x) - u^-(x - se_1)$, $x \in B_R(0)$, then $U_s \in \mathcal{M}$. 
But 
$$
\nu \leqslant E[U_s] = E[v_s - u^-(\cdot - se_1)] = E[v_s] + E[- u^-(\cdot - se_1)] < E[u^+] + E[- u^-] = E[u] = \nu
$$
in view of \eqref{eq:nonrad1}, which is impossible.

Consider now problem \eqref{D} in some annulus $B_{R_1}(0) \setminus \overline{B_{R_0}(0)}$. 
Suppose that this problem possesses a least energy nodal solution $u$ which is radial.
Hence, there exists $r \in (R_0,R_1)$ such that, without loss of generality, $u^+$ is a least energy positive solution of \eqref{D} in $B_{R_1}(0) \setminus \overline{B_r(0)}$ and $-u^-$ is a least energy negative solution of \eqref{D} in $B_r(0) \setminus \overline{B_{R_0}(0)}$. 
Shifting $B_r(0)$ along $e_1$ on a sufficiently small distance $s \geqslant 0$, we get a contradiction as above. Indeed, Proposition~\ref{prop:annulus}, together with the invariance of \eqref{D} upon orthogonal transformations of coordinates, implies that 
\begin{equation*}
\mu_+(B_{R_1}(0) \setminus \overline{B_r(se_1)}) < \mu_+(B_{R_1}(0) \setminus \overline{B_r(0)})
~\text{and}~
\mu_-(B_r(se_1) \setminus \overline{B_{R_0}(0)}) < \mu_-(B_{r}(0) \setminus \overline{B_{R_0}(0)}),
\end{equation*}
and corresponding minimizers generate a function from $\mathcal{M}$ which energy is strictly less than $\nu$. 
A contradiction.

\par
\bigskip
\noindent
{\bf Acknowledgments.} 
The first author was supported by the project LO1506 of the Czech Ministry of Education, Youth and Sports.
The second author was supported by the grant 17-01-00678 of Russian Foundation for Basic Research.
The second author wishes to thank the University of West Bohemia, where this research was started, for the invitation and hospitality.
The authors would like to thank A.I.\ Nazarov for stimulating discussions and valuable advices.

\appendix
\section{Appendix}\label{section:appendix}
For readers' convenience, in this section we sketchily show that the minimization problems $\mu_\pm(\Omega)$ and $\nu$ given by \eqref{mu:Nehari} and \eqref{nu} possess minimizers which are least energy constant-sign and nodal solutions of \eqref{D}, respectively.  
(Note that the existence of ``abstract'' constant-sign and nodal solutions for problems of the type \eqref{D} is known under much weaker assumptions on $f$, see, for instance, \cite{dinca,liuwang}. However, we are interested in solutions with the least energy property.) 
Throughout this section, we always assume that $(A_1)-(A_4)$ are fulfilled. 

First we need the following result about the geometry of the functional $E$.
\begin{lemma}\label{lem:superlin}
	Let $u \in \wolp(\Omega) \setminus \{0\}$. Then there exists a unique constant $\alpha(u) \in (0, +\infty)$ such that $\alpha(u) u \in \mathcal{N}(\Omega)$. 
	Moreover, $E[\alpha(u)u] = \max\limits_{\alpha > 0}E[\alpha u] > 0$.
\end{lemma}
\begin{proof}
	Define $Q: (0,+\infty) \to \mathbb{R}$ by $Q(\alpha) = E[\alpha u]$. Differentiating $Q$, we get
	\begin{equation*}\label{eq:Phi'}
	Q'(\alpha) = \alpha^{p-1}\left[\int_{\Omega} |\nabla u|^p \, dx - \int_{\Omega} |u|^p \frac{f(\alpha u)}{|\alpha u|^{p-2} \alpha u} \, dx \right].
	\end{equation*}
	Assume that there exists a critical point $\alpha_1 > 0$ of $Q$, i.e., $Q'(\alpha_1)=0$. Evidently, $\alpha_1 u \in \mathcal{N}(\Omega)$. 
	From the first part of $(A_3)$ we deduce that $\frac{f(s)}{|s|^{p-2}s}$ is strictly decreasing for $s<0$ and strictly increasing for $s > 0$. This implies that $Q'(\alpha) > 0$ for all $\alpha \in (0, \alpha_1)$, and $Q'(\alpha) < 0$ for all $\alpha > \alpha_1$. 
	Thus, any possible critical point of $Q$ on $(0, +\infty)$ is a point of a strict local maximum, and hence $Q$ has at most one critical point on $(0, +\infty)$. Moreover, since $Q(0) = 0$, we get $Q(\alpha_1) > 0$.
	
	Let us show that a critical point exists. 
	In view of $(A_1)-(A_3)$, we apply \cite[Theorem~17]{dinca} to deduce that $Q(\alpha) > 0$ for some $\alpha>0$ small enough. 
	On the other hand, due to $(A_1)$, $(A_2)$, and $(A_4)$, \cite[Proposition~7]{dinca} implies that $Q(\alpha) < 0$ for $\alpha > 0$ large enough. Therefore, since $Q(0) = 0$, there exists a positive zero of $Q'$.
\end{proof}

Arguing as in the first part of the proof of Lemma~\ref{lem:superlin}, we deduce the following fact.
\begin{lemma}\label{lem:superlin2}
	Let $u \in \wolp(\Omega)$ and let $\Omega_1$ be a subdomain of $\Omega$. 
	If $u \not\equiv 0$ a.e.\ in $\Omega_1$ and
	$$
	\int_{\Omega_1} |\nabla u|^p \, dx - \int_{\Omega_1} u f(u) \, dx = 0,
	$$
	then 
	$$
	\frac{1}{p}\int_{\Omega_1} |\nabla u|^p \, dx - \int_{\Omega_1} F(u) \, dx = 
	\max\limits_{\alpha>0}\left( \frac{1}{p}\int_{\Omega_1} |\nabla (\alpha u)|^p \, dx - \int_{\Omega_1} F(\alpha u) \, dx \right) > 0.
	$$
\end{lemma}

\medskip
\begin{lemma}\label{lem:existence}
	There exist minimizers of $\mu_\pm(\Omega)$ and $\nu$.
\end{lemma}
\begin{proof}
	Note first that $\mathcal{N}(\Omega)$ and $\mathcal{M}$ are nonempty. Indeed, let us take any nontrivial $u_1, u_2 \in \wolp(\Omega) \setminus \{0\}$ such that $u_1 \geqslant 0$, $u_2 \leqslant 0 $, and $u_1, u_2$ have disjoint supports. Then, Lemma~\ref{lem:superlin} yields the existence of $\alpha_1, \alpha_2 > 0$ such that $\alpha_1 u_1, \alpha_2 u_2 \in \mathcal{N}(\Omega)$. Hence, there exist minimizing sequences for $\mu_+(\Omega)$ and $\mu_-(\Omega)$. Moreover, $\alpha_1 u_1 + \alpha_2 u_2 \in \mathcal{M}$, which implies the existence of a minimizing sequence for $\nu$. 
	
	Let us prove that a minimizing sequence for $\nu$ converges. The cases of $\mu_\pm(\Omega)$ can be treated analogously.
	Let $\{u_k\}_{k \in \mathbb{N}} \subset \mathcal{M}$ be a minimizing sequence for $\nu$. 
	(The following technical details are reminiscent of the proof of Proposition~\ref{proof:contin}.)
	First we show that $\{u_k\}_{k \in \mathbb{N}}$ is bounded. 
	Note that from $(A_2)$, $(A_4)$, and the first part of $(A_3)$ we get 
	\begin{align*}
	&\int_{\Omega} F(u_k) \, dx \leqslant C_1 +  \frac{1}{\theta} \int_{\{x \in \Omega:\, |u_k(x)| > s_0\}} u_k f(u_k) \, dx 
	\leqslant C_1 +  \frac{1}{\theta} \int_{\Omega} u_k f(u_k) \, dx
	\end{align*}	
	where $C_1 > 0$ does not depend on $k$. 
	Therefore, supposing that $\|\nabla u_k\|_{L^p(\Omega)} \to +\infty$ as $k \to +\infty$ and recalling that $u_k \in \mathcal{M}$ for each $k \in \mathbb{N}$, we obtain
	\begin{align*}
	\notag
	E[u_k] 
	&= \frac{1}{p} \int_{\Omega} |\nabla u_k|^p \, dx - \int_{\Omega} F(u_k) \, dx\\
	&\geqslant \frac{1}{p} \int_{\Omega} |\nabla u_k|^p \, dx - \frac{1}{\theta} \int_{\Omega} u_k f(u_k) \, dx - C_1 =
	\left(\frac{1}{p}-\frac{1}{\theta}\right) \int_{\Omega} |\nabla u_k|^p \, dx - C_1 \to +\infty
	\end{align*}
	since $\theta > p$. However, this fact contradicts a minimization nature of $\{u_k\}_{k \in \mathbb{N}}$. Thus, there exists $u \in \wolp(\Omega)$ such that, up to a subsequence, $u_k \to u$ and $u_k^\pm \to u^\pm$ weakly in $\wolp(\Omega)$ and strongly in $L^{q}(\Omega)$, $q \in (p, p^*)$ (see \cite[Section~3]{cosio} and a direct generalization of \cite[Lemma~2.3]{cosio} to the case $p>1$). 
	At the same time, due to $(A_2)$ and the second part of $(A_3)$, we can find $\mu \in (0, \lambda_p(\Omega))$ and $C_2>0$ such that $|f(s)| \leqslant \mu |s|^{p-1} + C_2 |s|^{q-1}$ for all $s \in \mathbb{R}$.
	Therefore, we get
	$$
	\int_{\Omega} |\nabla u_k^+|^p \, dx = \int_{\Omega}  u_k^+ f(u_k^+) \, dx \leqslant \frac{\mu}{\lambda_p(\Omega)} \int_{\Omega} |\nabla u_k^+|^p \, dx + C_3 \left(\int_{\Omega} |\nabla u_k^+|^p \, dx \right)^\frac{q}{p}
	$$
	for some $C_3 > 0$. 
	If we suppose that $\|\nabla u_k^+\|_{L^p(\Omega)} \to 0$ as $k \to +\infty$, then for sufficiently large $k$ we obtain a contradiction since $\mu < \lambda_p(\Omega)$ and $q>p$.
	Thus, there exists $C_4 > 0$ such that $\|\nabla u_k^+\|_{L^p(\Omega)} > C_4$ for all $k$ large enough.
	Hence, $\int_{\Omega} u_k^+ f(u_k^+) \, dx > C_4$, which yields $u^+ \not\equiv 0$ in $\Omega$. 
	The same facts hold true for $-u_k^-$. 
	
	Now we show that $u^\pm_k \to u^\pm$ strongly in $\wolp(\Omega)$. By the weak convergence, we have
	\begin{equation}\label{eq:weakconv}
	\|\nabla u^\pm\|_{L^p(\Omega)} \leqslant \liminf\limits_{k \to +\infty} \|\nabla u^\pm_k\|_{L^p(\Omega)}.
	\end{equation} 
	Suppose, for instance, that $u^+_k$ does not converge strongly in $\wolp(\Omega)$, i.e., the strict inequality in \eqref{eq:weakconv} holds. 
	Then Lemma~\ref{lem:superlin} implies the existence of $\alpha(u^+), \alpha(u^-) > 0$ such that $\alpha(u^+) u^+ \in \mathcal{N}(\Omega)$ and $-\alpha(u^-) u^- \in \mathcal{N}(\Omega)$, and hence $\alpha(u^+) u^+ - \alpha(u^-) u^- \in \mathcal{M}$. Moreover, $\alpha(u^\pm)$ are unique points of maximum of $E[\alpha u^\pm]$ with respect to $\alpha>0$. 
	Since each $u_k \in \mathcal{M}$, we also deduce from Lemma~\ref{lem:superlin} that $\alpha=1$ is a unique point of maximum of both $E[\alpha u_k^+]$ and $E[\alpha u_k^-]$ with respect to $\alpha>0$. Therefore,
	\begin{align*}
	\nu 
	\leqslant E[\alpha(u^+) u^+ - \alpha(u^-) u^-] 
	&< \liminf\limits_{k \to +\infty} 
	\left( E[\alpha(u^+) u_k^+] +  E[-\alpha(u^-) u_k^-] \right) \\
	&\leqslant \liminf\limits_{k \to +\infty} \left( E[u_k^+] +  E[-u_k^-] \right) 
	= \liminf\limits_{k \to +\infty} E[u_k^+ - u_k^-]
	= \nu,
	\end{align*}
	a contradiction. Consequently, $u_k^\pm \to u^\pm$ strongly in $\wolp(\Omega)$, $u \in \mathcal{M}$, and $E[u] = \nu$.
\end{proof}

\begin{remark}
	As a corollary of Lemmas~\ref{lem:existence} and \ref{lem:superlin} we have $\mu_\pm(\Omega) > 0$ and $\nu > 0$.
\end{remark}

Generalizing directly the proof of \cite[Proposition~3.1]{BartschWethWillem} (see also \cite[Proposition~6.1]{BartschWethWillem}) to the case $p>1$, we obtain the following result.
\begin{lemma}\label{lem:critical}
	Any minimizers of $\mu_\pm(\Omega)$ and $\nu$ are critical points of $E$ on $\wolp(\Omega)$, that is, weak solutions of \eqref{D} with corresponding sign properties.
\end{lemma}

\begin{remark}\label{rem:regularity}
	Any weak solution of \eqref{D} belongs to $C^{1,\beta}(\overline{\Omega})$ for some $\beta \in (0,1)$, see \cite[Corollary~1.1]{guedda} and \cite{lieberman}. 
\end{remark}

\begin{remark}\label{rem:maximum}
	From the first part of $(A_3)$ it follows that $f(0) = 0$, $f(s) > 0$ for $s > 0$, and $f(s) < 0$ for $s<0$. Hence, applying the strong maximum principle \cite[Theorem~5]{vazquez}, we derive that any weak constant-sign solution of \eqref{D} is either strictly positive or strictly negative in $\Omega$, and has a nonzero normal derivative on the boundary $\partial \Omega$.
\end{remark}

\begin{remark}\label{rem:reg2}
	Let $u \in C^{1,\beta}(\overline{\Omega})$ be a positive weak solution of \eqref{D}. 
	If $|\nabla u| > \eta$ in $\Omega_\delta := \{x \in \Omega:\, \mathrm{dist}(x, \partial \Omega) < \delta\}$ for some $\eta, \delta>0$, then $u \in C^{2}(\overline{\Omega_\delta})$. 
	See, e.g., \cite[Lemma~5.2]{fleckingertakac} with the source function $\widetilde{f}(x) := f(u(x)) - a u(x)^{p-1}$, $a < 0$.
\end{remark}

\begin{remark}\label{rem:Omega_t}
	All the results stated above in Appendix~\ref{section:appendix} remain valid for problem \eqref{D} in perturbed domains $\Omega_t = \Phi_t(\Omega)$, where the deformation $\Phi_t$ is given by \eqref{Phi}, and $|t| < \delta$ with sufficiently small $\delta > 0$. 
	Indeed, the only assumption on the nonlinearity $f$ which depends on a domain is the second part of $(A_3)$.
	However, since $\lambda_p(\Omega_t)$ is continuous at $t=0$ (see \cite{garcia}), we can take $\delta > 0$ smaller (if necessary) and find $\widetilde{C}>0$ such that $\limsup\limits_{s \to 0}\frac{f(s)}{|s|^{p-2}s} < \widetilde{C} < \lambda_p(\Omega_t)$ for all $|t| < \delta$. That is, the second part of $(A_3)$ is satisfied uniformly for all $|t| < \delta$.
\end{remark}

\end{document}